\documentclass[11pt]{amsart}
\usepackage[a4paper,margin=28.5mm]{geometry}
\usepackage[T1]{fontenc}
\usepackage{lmodern}
\usepackage{amsmath,amssymb,amsthm,mathtools}
\usepackage{mathrsfs}
\usepackage{microtype}
\usepackage{enumitem}
\usepackage{booktabs}
\usepackage{xcolor}
\usepackage[colorlinks=true,linkcolor=blue!55!black,
 citecolor=blue!55!black,urlcolor=blue!55!black]{hyperref}
\newtheorem{theorem}{Theorem}[section]
\newtheorem{proposition}[theorem]{Proposition}
\newtheorem{lemma}[theorem]{Lemma}
\newtheorem{corollary}[theorem]{Corollary}
\theoremstyle{definition}
\newtheorem{definition}[theorem]{Definition}
\newtheorem{algorithm}[theorem]{Algorithm}

\theoremstyle{remark}
\newtheorem{remark}[theorem]{Remark}

\newcommand{\Br}{\operatorname{Br}}
\newcommand{\cl}[1]{\widehat{#1}}
\newcommand{\iso}{\cong}
\newcommand{\GT}{\mathcal T}
\newcommand{\ind}{\mathbf{1}}

\setlist[enumerate]{label=\textup{(\arabic*)},leftmargin=*,itemsep=4pt}
\setlist[itemize]{leftmargin=*,itemsep=3pt}
\numberwithin{equation}{section}
\allowdisplaybreaks[2]
\title[Polynomial-time algorithms for links]
{Polynomial-time algorithms for links}
\author{Thiago de Paiva}
\address[]{Beijing International Center for Mathematical Research, Peking University, Beijing 100871, China P.R.}
\email[]{thhiagodepaiva@gmail.com}
\author{Yi Liu}
\address{Beijing International Center for Mathematical Research, Peking University, Beijing 100871, China P.R.}
\email{liuyi@bicmr.pku.edu.cn}
\date{}
\subjclass[2020]{57K10, 20F36, 68Q25}
\keywords{Generalised T-link, braid word, polynomial-time algorithm, link invariant}
\hypersetup{pdfauthor={Thiago de Paiva and Yi Liu},
 pdftitle={Polynomial-time algorithms for links},
 pdfsubject={Explicit braid conversion, complexity bounds, and polynomial-time transfer}}

\begin{document}

\begin{abstract}
Every link in the \(3\)-sphere admits a generalised \(T\)-link
presentation. In this paper, we give an explicit polynomial-time algorithm
which, given a signed Artin braid word and the number of
strands, computes such a presentation. Combining this
construction with Yamada--Vogel braiding gives a polynomial-time
conversion from arbitrary oriented combinatorial link diagrams
to generalised \(T\)-link parameters. We prove polynomial bounds on the number of strands and
crossings in the resulting braid.

As an application, we prove that a link property or a link
invariant is computable in polynomial time from diagrams if
and only if the corresponding problem is solvable in polynomial
time from generalised \(T\)-link parameters, with parameter input measured by the total number of strands
and crossings in the corresponding braid diagram. Whenever a classification criterion
or an explicit invariant formula is available in these
parameters, it can be applied to the original link by
substituting the parameters obtained from the diagram.
A polynomial-time test or evaluation then gives a polynomial-time
procedure on diagrams. We also establish the corresponding transfer for approximating
real-valued invariants to prescribed accuracy.
\end{abstract}

\maketitle

\section{Introduction}\label{sec:introduction}

%Ghrist proved that certain Lorenz-like templates are universal:
%each carries a representative of every link in \(S^3\)
%\cite[Theorem~3 and Corollary~3]{Ghrist1997}.
%His proof is constructive: it starts from an arbitrary braid
%and places an isotopic copy of its closure on a universal template.
%De Paiva, Hui, and Rodr\'iguez Migueles established an explicit
%correspondence between links carried by the corresponding
%Lorenz-like templates and generalised \(T\)-link presentations
%\cite[Theorem~3.10 and Corollary~3.13]
%{dePaivaHuiRodriguez}.
%Together, these results give a constructive route from an arbitrary
%braid to a generalised \(T\)-link presentation.

%In Subsection~\ref{subsec:ghrist-comparison}, we analyse
%Ghrist's successive-ear construction, with the ear placement
%specified there and the resulting symbolic itineraries written
%in full. For a family of three-strand inputs, this procedure
%produces exponentially long itineraries. Consequently, the
%construction with this expanded output is not polynomial-time
%in the input word length. 

Generalised \(T\)-links represent all links in \(S^3\)
\cite[Corollary~3.13]{dePaivaHuiRodriguez}.
Their general definition includes an integer \(d\geq0\), called
the offset, which records \(d\) additional split unknotted
components \cite[Definition~3.8]{dePaivaHuiRodriguez}.
Throughout this paper, we use presentations with \(d=0\).

Set
\[
 \delta_r=\sigma_1\sigma_2\cdots\sigma_{r-1}.
\]
With this convention, a generalised \(T\)-link is the closure,
on exactly \(r_k\) strands, of
\begin{equation}\label{eq:intro-braid}
 \delta_{r_1}^{s_1}\delta_{r_2}^{s_2}\cdots\delta_{r_k}^{s_k},
 \qquad 2\leq r_1<\cdots<r_k,
\end{equation}
where all exponents before the last one are positive, and the
last exponent is an integer different from \(-1,0,1\).
When the last exponent is positive, this is the ordinary
\(T\)-link form introduced by Birman and Kofman in their study
of Lorenz links \cite{BirmanKofman}.

This result gives a common parameter description for families
of knots and links studied in different settings. It raises
the following computational question: given a signed Artin
braid word and the number of strands, can we compute, in
polynomial time, generalised \(T\)-link parameters representing
its closure, and how large is the resulting presentation?

In this paper, we give an explicit algorithm which takes an
Artin braid word together with its number of strands and
returns generalised \(T\)-link parameters representing its
oriented closure. We prove polynomial bounds on both its
running time and the number of strands and crossings in the
resulting braid, in terms of the input strand number and
word length.

Our input consists of a strand number \(p\) and an expanded
Artin word
\[
 \beta=\sigma_{i_1}^{\epsilon_1}\cdots
        \sigma_{i_c}^{\epsilon_c}\in\Br_p,
 \qquad \epsilon_j\in\{1,-1\}.
\]
The integer \(c\) is the number of letters in this word.
In particular, it is the crossing number of the specified
braid diagram. The strand number \(p\) is part of the input
because an Artin word alone does not specify how many strands
are present. The same word may be regarded as a braid on more
strands by adding trivial strands on the right; each additional
strand contributes a split unknotted component to the closure.
Our conversion preserves the entire oriented link, including
these components, and produces a presentation with \(d=0\)
even when the input link is split.

Our main result is the following.

\begin{theorem}\label{thm:main}
There is a deterministic algorithm which, given a finite Artin
word and the number of strands, returns a generalised
\(T\)-link presentation of its oriented closure.
For \(p\geq3\) and \(c\geq1\), write the output as
\[
 \GT(\beta)=((r_1,s_1),\ldots,(r_k,s_k)),\qquad M=r_k.
\]
Then
\begin{equation}\label{eq:main-size}
 k\leq3c,\qquad M\leq\frac{pc(c+3)}2.
\end{equation}
All strand parameters increase strictly, every exponent before
the last is positive, and \(|s_k|\geq2\).

For these inputs, the algorithm computes the parameters,
written in binary, using \(O(c)\) integer operations and
storing \(O(c)\) integers. Put
\begin{equation}\label{eq:L}
 L=\left\lceil\log_2\bigl(2+p(c+1)\bigr)\right\rceil.
\end{equation}
In the computational model specified in
Section~\ref{subsec:model}, the computation takes
\(O(cL+L^2)\) bit operations and uses \(O(cL)\) bits of storage.
\end{theorem}

The main construction is based on repeated applications, in the
reverse direction, of the negative braid-reduction isotopy of
de Paiva and Liu \cite[Proposition~3.1]{dePaivaLiu}.
We first rewrite a conjugate of the input braid as a product of
negative prefix blocks followed by a central positive full twist.
In the first round, repeated applications of the isotopy give
twisting blocks supported on increasing numbers of strands.
We then separate negative full twists from the intermediate
negative factors and apply the same isotopy repeatedly in a
second round. This gives a generalised \(T\)-link presentation
whose only negative twisting exponent is the last one.
For completeness, we give an algebraic proof of the strand-removal
step and verify the support conditions at each application.
The parameter formulas record these operations and allow us to
compute the output directly from the input word.

After decomposing the braid into blocks, we may move initial
blocks to the end by conjugation. Different choices of the first
block can lead to different final strand numbers. We choose the
one that gives the smallest number within this construction,
without claiming minimality among all generalised \(T\)-link
presentations of the same link.

The input braid is given letter by letter as an Artin word of
length \(c\), while the output is a list of parameter pairs
describing the final braid. We compute these pairs directly,
without writing out the final braid word or performing every
Markov move used in the proof. For fixed \(p\),
Theorem~\ref{thm:main} gives \(O(c)\) integer operations and
\(O(c\log(c+1))\) bit operations, even though the braid described
by the output may have quadratically many strands and cubically
many crossings in \(c\).

We also consider oriented link diagrams. We first apply Yamada--Vogel braiding \cite{Yamada,Vogel}
to obtain a braid whose closure is isotopic to the given
oriented link, and then apply our algorithm. For a diagram
with \(n\) crossings and \(\mu\) components, put \(N=n+\mu\).
The braiding step produces \(O(N)\) strands and a word of length
\(O(N^2)\). Theorem~\ref{thm:main} therefore gives a generalised
\(T\)-link presentation with \(O(N^2)\) pairs and \(O(N^5)\)
strands. The complete procedure can be implemented using
\(O(N^4\log(N+1))\) bit operations; see
Theorem~\ref{thm:diagram}. We include the number of components
because a diagram with no crossings can still contain
arbitrarily many disjoint circles.

The conversion also allows us to use classifications and invariant
formulas for generalised \(T\)-links to study arbitrary link diagrams.
To use such a classification or formula, we must include the cost
of testing the classification criterion or evaluating the invariant
formula. We make this relation precise in the next theorem.

For a generalised \(T\)-link presentation
\(\mathcal P=((r_1,s_1),\ldots,(r_k,s_k))\), put
\begin{equation}\label{eq:intro-expanded-size}
 W(\mathcal P)=r_k+\sum_{j=1}^k(r_j-1)|s_j|.
\end{equation}
This is the sum of the number of strands and the number of crossings
in the corresponding braid diagram. A time bound polynomial in \(W\)
is equivalent to polynomial time when the parameters are supplied
in unary; the proof and the distinction from binary encoding are
given in Section~\ref{subsec:transfer-sizes}.

For invariants, we fix in advance how each value is represented
as a finite sequence of symbols from a fixed finite alphabet.
Examples include integers written in binary and polynomials
with integer coefficients represented by lists of coefficients
and exponents. The same output representation is used for
diagram and parameter inputs, and the running time includes
the time needed to write the complete answer.

\begin{theorem}\label{thm:intro-transfer}
Let \(\chi\) be a fixed property of nonempty oriented links
in \(S^3\) that is preserved by ambient isotopy.
The following are equivalent:
\begin{enumerate}
\item There is a deterministic algorithm that, given a
combinatorial diagram \(D\) with \(n\) crossings and \(\mu\)
components, decides whether \(L(D)\) has property \(\chi\)
in time polynomial in \(n+\mu\).

\item There is a deterministic algorithm that, given any
generalised \(T\)-link parameter list \(\mathcal P\), decides
whether \(T(\mathcal P)\) has property \(\chi\) in time
polynomial in \(W(\mathcal P)\).
\end{enumerate}
The same equivalence holds for computing any fixed link
invariant with the output representation specified above.
\end{theorem}

To prove this theorem, we convert between diagrams and parameter
lists. From a diagram with \(n\) crossings and \(\mu\) components,
our algorithm produces a generalised \(T\)-link parameter list
\(\mathcal P\) satisfying \(W(\mathcal P)=O((n+\mu)^8)\). Conversely, writing out the braid described
by a parameter list and closing it gives a diagram whose total
number of crossings and components is \(O(W)\). These conversions
take polynomial time in \(n+\mu\) and \(W\), respectively, and
preserve the oriented link. An algorithm for deciding a link
property or computing an invariant from either description can
therefore be used with the other description after conversion.
The argument does not require a unique parameter list for each
link: every list representing the same oriented link must give
the same answer.

\begin{corollary}\label{cor:intro-binary-transfer}
Suppose that a fixed link property can be decided, or a fixed
link invariant with the output representation specified above
can be computed, by a single deterministic algorithm in time
polynomial in the binary length of a generalised \(T\)-link
parameter list. Then the corresponding problem can be solved
in polynomial time from combinatorial link diagrams.
It suffices that the parameter algorithm be correct and satisfy
the stated running-time bound on the lists produced by our
conversion.
\end{corollary}

Our intended application is to use exact formulas expressing
link invariants directly in terms of generalised \(T\)-link
parameters. Suppose that an invariant \(I\) admits an explicit
formula
\[
 I(T(\mathcal P))
 =F(r_1,s_1,\ldots,r_k,s_k).
\]
Given an arbitrary oriented link diagram \(D\), our conversion
produces a parameter list \(\mathcal P\) representing the same
oriented link. Substituting these parameters into the formula
then gives \(I(L(D))\). For example, a formula for the Jones
polynomial in these parameters would give the Jones polynomial
of the original link by this substitution. The same principle
applies to a homological invariant with a fixed finite output
representation.

Likewise, a complete classification expressed by necessary and
sufficient conditions on the parameters gives a criterion for
the corresponding property. If testing these conditions or
evaluating the invariant formula takes time polynomial in the
binary length of the parameter list, the corollary gives a
polynomial-time procedure on diagrams. For this application,
it suffices to establish these guarantees on the lists produced
by our conversion.

In the final section, we also consider additional inputs and
real-valued invariants. For a real-valued invariant \(V\), an
integer precision request \(b\geq1\) asks for a rational
approximation \(\widetilde V\) satisfying
\[
 |\widetilde V-V(L)|\leq 2^{-b}.
\]
Thus, \(b\) specifies the permitted absolute error. For example,
\(b=10\) permits an error of at most \(1/1024\), and increasing
\(b\) by one halves the permitted error. We include the requested
precision in the size measures, using \(n+\mu+b\) for diagrams
and \(W(\mathcal P)+b\) for parameter lists. With these measures,
we prove the analogue of Theorem~\ref{thm:intro-transfer} for
polynomial-time approximation. Hyperbolic volume is one example:
we assume that the input represents a hyperbolic link \(L\) and
seek an approximation to \(\operatorname{Vol}(S^3\setminus L)\)
with the requested error bound.

The paper is organised as follows.
Section~\ref{sec:background} fixes our braid conventions, defines
generalised \(T\)-link presentations, and specifies the computational
model. Section~\ref{sec:isotopy} establishes the braid identities
and reduction isotopies used in the construction.
Section~\ref{sec:cyclic} gives the decomposition into negative
prefix blocks. Section~\ref{sec:parameters} derives the explicit
output parameters and determines which choice of the first block
minimises the number of strands within this construction.
Section~\ref{sec:shortcuts} combines the construction with elementary
simplification and direct recognition to give the complete algorithm.
Section~\ref{sec:complexity} proves the running-time, storage, and
output-size bounds, and Section~\ref{sec:examples} illustrates
the algorithm with explicit examples.
Section~\ref{sec:diagrams} extends the conversion to combinatorial
link diagrams. Finally, Section~\ref{sec:transfer} proves the transfer results
for link properties, finitely encoded invariants, and real-valued
invariants with prescribed accuracy.

\section{Braid conventions and input size}\label{sec:background}

In this section we recall our braid notation and generalised link
presentations, and specify the input and computational model. We
include the conventions needed for split links and short braid words.

\subsection{Twisting blocks and Markov moves}
\label{subsec:braids}

We begin by recalling the braid-group conventions and the elementary
operations used throughout the paper. Unless otherwise stated, the
standard inclusion \(\Br_p\hookrightarrow\Br_{p+1}\) is obtained by
adding one trivial strand on the right.

For \(p\geq2\), the Artin braid group \(\Br_p\) is generated by
\(\sigma_1,\ldots,\sigma_{p-1}\), subject to the relations
\[
 \sigma_i\sigma_j=\sigma_j\sigma_i
 \quad (|i-j|\geq2),
 \qquad
 \sigma_i\sigma_{i+1}\sigma_i
 =
 \sigma_{i+1}\sigma_i\sigma_{i+1},
\]
where the first relation holds for \(1\leq i,j\leq p-1\)
and the second for \(1\leq i\leq p-2\).
We also set \(\Br_1=\{1\}\).
We write \(\cl{\alpha}\) for the usual oriented closure of a braid
\(\alpha\), and use \(\iso\) to denote ambient isotopy of oriented
links in \(S^3\).

We recall four types of operations used throughout the paper.
When we need to specify the strand number explicitly, we write
\(\operatorname{cl}_p(\alpha)\) for the oriented closure of
\(\alpha\in\Br_p\) on exactly \(p\) strands.

\begin{enumerate}
\item \textbf{Braid equalities.}
These are equalities between words representing the same element
of a fixed braid group. They follow from the Artin relations,
cancellation of adjacent inverse letters, and algebraic consequences
of these identities. Replacing a word or a factor by an equal braid
preserves both the braid element and its declared strand number.
Consequently, it also preserves the oriented closure.

\item \textbf{Conjugation and cyclic permutations.}
For \(\alpha,\gamma\in\Br_p\), conjugation preserves the oriented
closure:
\[
 \operatorname{cl}_p(\alpha)
 \iso
 \operatorname{cl}_p(\gamma^{-1}\alpha\gamma).
\]
In particular, if \(X,Y\in\Br_p\), then
\[
 X^{-1}(XY)X=YX,
\]
and hence
\[
 \operatorname{cl}_p(XY)
 \iso
 \operatorname{cl}_p(YX).
\]
Thus, a factor at the beginning of a product may be moved to the
end, or conversely, without changing the oriented link represented
by its closure. Here \(X\) and \(Y\) may themselves be products of
several braid factors. We refer to this operation as a
\emph{cyclic permutation of factors}, or a \emph{cyclic conjugation}.
It does not change the number of strands.

\item \textbf{Markov stabilisation.}
Let \(\alpha\in\Br_p\), and regard it as an element of
\(\Br_{p+1}\) by adding one trivial strand on the right.
For \(\epsilon\in\{1,-1\}\), the move
\[
 (p,\alpha)
 \longmapsto
 (p+1,\alpha\sigma_p^\epsilon)
\]
is called a \emph{Markov stabilisation}. It is positive when
\(\epsilon=1\) and negative when \(\epsilon=-1\).
Geometrically, the move adds one strand and one crossing involving
the last two strands. It preserves the oriented closure:
\[
 \operatorname{cl}_p(\alpha)
 \iso
 \operatorname{cl}_{p+1}(\alpha\sigma_p^\epsilon).
\]

\item \textbf{Markov destabilisation.}
This is the inverse move:
\[
 (p+1,\alpha\sigma_p^\epsilon)
 \longmapsto
 (p,\alpha),
 \qquad \alpha\in\Br_p.
\]
It is positive or negative according to whether
\(\epsilon=1\) or \(\epsilon=-1\).
The condition \(\alpha\in\Br_p\) is essential: in the displayed
representative, all factors preceding \(\sigma_p^\epsilon\)
are supported on the first \(p\) strands. The last strand therefore
participates only in the displayed crossing, and removing that
strand together with the crossing gives \(\alpha\) on \(p\) strands.
Thus,
\[
 \operatorname{cl}_{p+1}(\alpha\sigma_p^\epsilon)
 \iso
 \operatorname{cl}_p(\alpha).
\]
Whenever we apply a destabilisation below, we verify this support
condition explicitly.
\end{enumerate}

The full Markov theorem states that two braids have ambient
isotopic oriented closures if and only if they are related by
a finite sequence of braid equalities, conjugations, and signed
Markov stabilisations and destabilisations.
In the arguments below, we use only the fact that each of these
operations preserves the oriented link type.
We refer to \cite{KasselTuraev} for the full statement of
Markov's theorem.

We now introduce two braid words used to describe twisting blocks.
For every \(r\geq2\), define
\begin{equation}\label{eq:roots}
 \delta_r=\sigma_1\cdots\sigma_{r-1},
 \qquad
 E_r=\sigma_1^{-1}\cdots\sigma_{r-1}^{-1}.
\end{equation}
The order of the generators in the definition of \(E_r\) is
important. Indeed,
\[
 \delta_r^{-1}
 =
 \sigma_{r-1}^{-1}\cdots\sigma_1^{-1}.
\]
For \(r=2\), we have \(E_2=\delta_2^{-1}\). For \(r\geq3\), however,
\(E_r\) and \(\delta_r^{-1}\) are distinct braid elements, although
they are conjugate. Thus, \(E_r\) is obtained by taking the negative
generators in increasing order, rather than by taking the inverse
of the word defining \(\delta_r\).

For every \(r\geq2\), let \(\Delta_r\in\Br_r\) denote the standard
Garside half twist, given by
\[
 \Delta_r
 =
 \delta_r\delta_{r-1}\cdots\delta_2
 =
 (\sigma_1\cdots\sigma_{r-1})
 (\sigma_1\cdots\sigma_{r-2})\cdots\sigma_1,
\]
where each \(\delta_j\) is regarded as an element of \(\Br_r\)
by adding trivial strands on the right. Reversing the order of
the factors and taking their inverses gives
\[
 \Delta_r^{-1}
 =
 \delta_2^{-1}\delta_3^{-1}\cdots\delta_r^{-1}
 =
 \sigma_1^{-1}
 (\sigma_2^{-1}\sigma_1^{-1})\cdots
 (\sigma_{r-1}^{-1}\cdots\sigma_1^{-1}).
\]
For \(r=2\), these formulas reduce to
\(\Delta_2=\sigma_1\) and \(\Delta_2^{-1}=\sigma_1^{-1}\).
The squares \(\Delta_r^2\) and \(\Delta_r^{-2}\) are, respectively,
the positive and negative full twists on \(r\) strands. They satisfy
\[
 \delta_r^r=\Delta_r^2,
 \qquad
 \delta_r^{-r}=\Delta_r^{-2}.
\]

To obtain the corresponding identity for \(E_r\), recall that
conjugation by \(\Delta_r\) reverses the indices of the Artin
generators:
\[
 \Delta_r\sigma_i\Delta_r^{-1}
 =
 \sigma_{r-i},
 \qquad 1\leq i\leq r-1.
\]
The index-reversal identity also holds for inverse generators,
since
\[
 \Delta_r\sigma_i^{-1}\Delta_r^{-1}
 =
 (\Delta_r\sigma_i\Delta_r^{-1})^{-1}
 =
 \sigma_{r-i}^{-1},
 \qquad 1\leq i\leq r-1.
\]
Consequently, inserting \(\Delta_r^{-1}\Delta_r=1\) between
successive generators, we obtain
\[
 \begin{aligned}
 \Delta_r\delta_r^{-1}\Delta_r^{-1}
 &=
 \Delta_r
 (\sigma_{r-1}^{-1}\cdots\sigma_1^{-1})
 \Delta_r^{-1}\\
 &=
 (\Delta_r\sigma_{r-1}^{-1}\Delta_r^{-1})
 \cdots
 (\Delta_r\sigma_1^{-1}\Delta_r^{-1})\\
 &=
 \sigma_1^{-1}\cdots\sigma_{r-1}^{-1}\\
 &=
 E_r.
 \end{aligned}
\]
Taking the \(r\)-th power and using
\(\delta_r^r=\Delta_r^2\), we obtain
\[
 \begin{aligned}
 E_r^r
 &=
 \Delta_r\delta_r^{-r}\Delta_r^{-1}\\
 &=
 \Delta_r\Delta_r^{-2}\Delta_r^{-1}\\
 &=
 \Delta_r^{-2}.
 \end{aligned}
\]
Thus, although \(E_r\) and \(\delta_r^{-1}\) are generally different
braid elements, their \(r\)-th powers agree:
\begin{equation}\label{eq:full-twist}
 \delta_r^r=\Delta_r^2,
 \qquad
 E_r^r=\Delta_r^{-2}=\delta_r^{-r}.
\end{equation}

Both \(\Delta_r^2\) and \(\Delta_r^{-2}\) are central elements
of \(\Br_r\); see \cite{KasselTuraev}. In particular, these full
twists commute with every braid on the same \(r\) strands.
We emphasize, however, that centrality is being asserted in
\(\Br_r\). After embedding \(\Br_r\) into a larger braid group,
the same elements need not remain central. Accordingly, whenever
we commute a full-twist block with another braid factor, we will
check the support conditions that justify the commutation in
the appropriate braid subgroup.

\subsection{Generalised \texorpdfstring{\(T\)}{T}-link presentations}
\label{subsec:gt}

We now recall the offset-zero presentations in
\cite[Definition 3.8]{dePaivaHuiRodriguez}. We keep track of the
strand number in this definition so that unused input strands
remain part of the closure.

\begin{definition}\label{def:gt}
A generalised \(T\)-link presentation with offset zero is a
nonempty list
\[
 \mathcal P=((r_1,s_1),\ldots,(r_k,s_k))
\]
such that \(2\leq r_1<\cdots<r_k\), \(s_j>0\) for \(j<k\),
and \(s_k\in\mathbb Z\setminus\{-1,0,1\}\).
It represents the link
\[
 T(\mathcal P)
 =
 \cl{\delta_{r_1}^{s_1}\cdots\delta_{r_k}^{s_k}}
\]
on exactly \(M=r_k\) strands. We call \(k\) its number of pairs
and \(M\) its final strand parameter.
\end{definition}

The general definition permits an offset \(d\geq0\), replaces each
\(\sigma_i\) by \(\sigma_{d+i}\), and takes the closure on
\(d+r_k\) strands. This adds \(d\) split unknots.
Our output always has \(d=0\), including when the input is split.
The absence of an offset does not imply that the represented
link is nonsplit.

For a presentation \(\mathcal P\), its writhe and its expanded
Artin length are, respectively,
\begin{equation}\label{eq:metrics-definition}
 w(\mathcal P)=\sum_{j=1}^k(r_j-1)s_j,
 \qquad
 C(\mathcal P)=\sum_{j=1}^k(r_j-1)|s_j|.
\end{equation}
Both quantities refer to the specified braid representative.
Thus, \(C(\mathcal P)\) counts its crossings, without asserting
that this number is minimal for the link.

\subsection{Expanded input and the computational model}
\label{subsec:model}

We next specify how the input and output are written when counting
operations. In particular, we distinguish an expanded word from a
word whose powers are encoded in binary.
The input consists of a positive integer \(p\), in binary, and
a list of \(c\) signed integers \(\epsilon_j i_j\), where
\(\epsilon_j\in\{1,-1\}\) and \(1\leq i_j<p\).
Each entry denotes one Artin letter \(\sigma_{i_j}^{\epsilon_j}\).
A power \(\sigma_i^a\) therefore contributes \(|a|\) entries.
The declaration of \(p\) includes every unused strand.

The output is the binary list of pairs, with the usual delimiters.
Its Artin expansion is not part of the parameter-computation
problem. We count additions, subtractions, comparisons, Boolean
operations, array accesses, products, and Euclidean divisions on
the integers encountered by the algorithm.

For integer-operation bounds, these operations have unit cost.
More precisely, the results can also be read as word-RAM bounds
with word length a sufficiently large constant multiple of \(L\),
where \(L\) is defined in \eqref{eq:L}.
For bit-cost bounds we retain random access and charge \(O(L)\)
for an operation on \(O(L)\)-bit integers involving addition,
comparison, a Boolean operation, or a memory read or write.
Schoolbook multiplication and division cost \(O(L^2)\).
We call this the \emph{bit-cost random-access model}.
We do not identify this bound with an unanalysed simulation on
a sequential single-tape machine.

All constants in our asymptotic statements are uniform in \(p\)
and \(c\). In particular, the implementation never allocates
an array of length \(p\) merely to record unused strands.
For a succinct input with powers encoded in binary, the same
algorithm may take time polynomial in the expanded length
rather than in the succinct input size.

\section{The negative prefix isotopy}\label{sec:isotopy}

In this section we convert negative prefix blocks into twisting
blocks on increasing numbers of strands. We first express the
signed generators using \(E_r\), then prove the one-strand
reduction and iterate it to obtain the prefix formula.
The input widths may occur in any order.

\subsection{Identities for a signed generator}
\label{subsec:letters}

We first record identities for the signed generators in a fixed
braid group. By the Artin relations, we have
\begin{equation}\label{eq:shift}
 \delta_p\sigma_i^{\pm1}\delta_p^{-1}
 =\sigma_{i+1}^{\pm1},
 \qquad
 E_p\sigma_i^{\pm1}E_p^{-1}
 =\sigma_{i+1}^{\pm1}
 \quad (1\leq i\leq p-2).
\end{equation}
For the first equality, commute the generators at distance at
least two, apply the three-term braid relation to
\(\sigma_i,\sigma_{i+1}\), and cancel.
The second equality follows by applying the mirror automorphism
\(\sigma_j\mapsto\sigma_j^{-1}\) to the first one.

\begin{lemma}\label{lem:letters}
For \(p\geq3\) and \(1\leq i<p\),
\begin{equation}\label{eq:letter-identities}
 \sigma_i=E_p^iE_{p-1}E_p^{-i-1},
 \qquad
 \sigma_i^{-1}=E_p^{i-1}E_2E_p^{-(i-1)}.
\end{equation}
\end{lemma}

\begin{proof}
Repeated application of \eqref{eq:shift} gives
\[
 E_p^{i-1}\sigma_1^{\pm1}E_p^{-(i-1)}
 =\sigma_i^{\pm1},
 \qquad 1\leq i<p.
\]
Every application is valid because the index being shifted is
at most \(p-2\). Since \(E_2=\sigma_1^{-1}\), the negative-sign
case proves the second equality.

To prove the first equality, conjugate each generator of
\(E_{p-1}\) by \(E_p\). Using \eqref{eq:shift}, we obtain
\[
 \begin{aligned}
 E_pE_{p-1}E_p^{-1}
 &=(E_p\sigma_1^{-1}E_p^{-1})\cdots
   (E_p\sigma_{p-2}^{-1}E_p^{-1})\\
 &=\sigma_2^{-1}\cdots\sigma_{p-1}^{-1}\\
 &=\sigma_1
   (\sigma_1^{-1}\sigma_2^{-1}\cdots\sigma_{p-1}^{-1})\\
 &=\sigma_1E_p.
 \end{aligned}
\]
Multiplying on the right by \(E_p^{-1}\) gives
\[
 \sigma_1=E_pE_{p-1}E_p^{-2}.
\]
Finally, conjugating this identity by \(E_p^{i-1}\) yields
\[
 \begin{aligned}
 \sigma_i
 &=E_p^{i-1}\sigma_1E_p^{-(i-1)}\\
 &=E_p^{i-1}
   (E_pE_{p-1}E_p^{-2})E_p^{-(i-1)}\\
 &=E_p^iE_{p-1}E_p^{-i-1},
 \end{aligned}
\]
as required.
\end{proof}

\subsection{Removing one strand}\label{subsec:one-strand}

We now write a negative twisting block so that its last-strand
generator appears exactly once. This allows us to identify the
Markov destabilisation in the reduction.

\begin{lemma}\label{lem:power}
If \(m\geq3\) and \(1\leq q<m\), then
\begin{equation}\label{eq:power}
 \delta_m^q
 =
 \delta_{m-1}^q
 (\sigma_{m-1}\sigma_{m-2}\cdots\sigma_{m-q}).
\end{equation}
\end{lemma}

\begin{proof}
For \(n\geq2\) and \(1\leq q\leq n\), put
\[
 A_{n,q}
 =
 \sigma_n\sigma_{n-1}\cdots\sigma_{n-q+1}
 \in\Br_{n+1}.
\]
We first prove the auxiliary identity
\begin{equation}\label{eq:shuffle}
 A_{n,q}\delta_n\sigma_n
 =
 \delta_n A_{n,q+1},
 \qquad 1\leq q<n.
\end{equation}
We use induction on \(q\), proving the identity for every \(n>q\)
at each step.

For \(q=1\), the generator \(\sigma_n\) commutes with
\(\sigma_1,\ldots,\sigma_{n-2}\). Hence the Artin braid relation
gives
\[
 \begin{aligned}
 A_{n,1}\delta_n\sigma_n
 &=
 \sigma_n(\sigma_1\cdots\sigma_{n-2})
 \sigma_{n-1}\sigma_n\\
 &=
 (\sigma_1\cdots\sigma_{n-2})
 (\sigma_n\sigma_{n-1}\sigma_n)\\
 &=
 (\sigma_1\cdots\sigma_{n-2})
 (\sigma_{n-1}\sigma_n\sigma_{n-1})\\
 &=
 \delta_n A_{n,2}.
 \end{aligned}
\]
When \(n=2\), the product
\(\sigma_1\cdots\sigma_{n-2}\) is empty and is interpreted as
the identity, so this calculation also covers that case.

Now suppose that \(2\leq q<n\) and that \eqref{eq:shuffle}
has been proved with \(q-1\) in place of \(q\), for every
admissible value of \(n\).
Since \(1\leq q-1<n-1\), the induction hypothesis applies
with indices \(n-1\) and \(q-1\), yielding
\[
 A_{n-1,q-1}\delta_{n-1}\sigma_{n-1}
 =
 \delta_{n-1}A_{n-1,q}.
\]
Using
\[
 A_{n,q}=\sigma_n A_{n-1,q-1},
 \qquad
 \delta_n=\delta_{n-1}\sigma_{n-1},
\]
we obtain
\[
 \begin{aligned}
 A_{n,q}\delta_n\sigma_n
 &=
 \sigma_n
 \bigl(A_{n-1,q-1}\delta_{n-1}\sigma_{n-1}\bigr)
 \sigma_n\\
 &=
 \sigma_n\delta_{n-1}
 (\sigma_{n-1}\cdots\sigma_{n-q})\sigma_n\\
 &=
 \delta_{n-1}
 (\sigma_n\sigma_{n-1}\sigma_n)
 (\sigma_{n-2}\cdots\sigma_{n-q})\\
 &=
 \delta_{n-1}
 (\sigma_{n-1}\sigma_n\sigma_{n-1})
 (\sigma_{n-2}\cdots\sigma_{n-q})\\
 &=
 \delta_n A_{n,q+1}.
 \end{aligned}
\]
Here we moved \(\sigma_n\) past the generators of
\(\delta_{n-1}\) and past
\(\sigma_{n-2},\ldots,\sigma_{n-q}\), all of whose indices
differ from \(n\) by at least two. We then applied the braid
relation to \(\sigma_n\sigma_{n-1}\sigma_n\).
This proves \eqref{eq:shuffle}.

We now prove \eqref{eq:power} by induction on \(q\).
For \(q=1\), the identity is simply
\[
 \delta_m=\delta_{m-1}\sigma_{m-1}.
\]
Suppose that \(1\leq q<m-1\) and that \eqref{eq:power}
holds for \(q\). In terms of the auxiliary notation,
the induction hypothesis reads
\[
 \delta_m^q=\delta_{m-1}^q A_{m-1,q}.
\]
Since \(q<m-1\), we may apply \eqref{eq:shuffle} with
\(n=m-1\). Therefore,
\[
 \begin{aligned}
 \delta_m^{q+1}
 &=
 \delta_m^q\delta_m\\
 &=
 \delta_{m-1}^q A_{m-1,q}
 \delta_{m-1}\sigma_{m-1}\\
 &=
 \delta_{m-1}^{q+1}A_{m-1,q+1}\\
 &=
 \delta_{m-1}^{q+1}
 (\sigma_{m-1}\sigma_{m-2}\cdots\sigma_{m-q-1}).
 \end{aligned}
\]
This is precisely \eqref{eq:power} with \(q+1\) in place of
\(q\), completing the induction.
\end{proof}

\begin{lemma}\label{lem:one-strand}
Let \(2\leq q<m\) and \(B\in\Br_{m-1}\). Then
\begin{equation}\label{eq:one-strand}
 \cl{B\delta_m^{-q}}
 \iso
 \cl{E_qB\delta_{m-1}^{-q}},
\end{equation}
where the closures on the left and right are taken on \(m\)
and \(m-1\) strands, respectively.
The isotopy uses one negative Markov destabilisation and
cyclic conjugations, together with braid equalities.
\end{lemma}

\begin{proof}
We regard \(\Br_{m-1}\) as the subgroup of \(\Br_m\)
supported on the first \(m-1\) strands.

Set
\[
 D=\delta_{m-1},
 \qquad
 C=\sigma_{m-q}^{-1}\cdots\sigma_{m-2}^{-1},
 \qquad
 a=\sigma_{m-1}^{-1}.
\]
We first rewrite \(\delta_m^{-q}\) in a form containing
exactly one occurrence of the last-strand generator
\(\sigma_{m-1}^{-1}\).
By \eqref{eq:power},
\[
 \delta_m^q
 =
 D^q(\sigma_{m-1}\sigma_{m-2}\cdots\sigma_{m-q}).
\]
Taking inverses reverses the order of the factors, and therefore
\begin{equation}\label{eq:isolation}
 \delta_m^{-q}
 =
 (\sigma_{m-q}^{-1}\cdots\sigma_{m-2}^{-1})
 \sigma_{m-1}^{-1}D^{-q}
 =
 CaD^{-q}.
\end{equation}

We next establish the identity that will move \(C\) past
\(D^{-q}\).
Put \(t=m-1-q\geq0\).
By \eqref{eq:shift}, with \(p=m-1\), we have
\[
 D\sigma_i^{-1}D^{-1}=\sigma_{i+1}^{-1},
 \qquad 1\leq i\leq m-3.
\]
Applying it \(t\) times to each generator of
\(E_q=\sigma_1^{-1}\cdots\sigma_{q-1}^{-1}\) gives
\[
 D^tE_qD^{-t}
 =
 \sigma_{1+t}^{-1}\cdots\sigma_{q-1+t}^{-1}
 =
 \sigma_{m-q}^{-1}\cdots\sigma_{m-2}^{-1}
 =
 C.
\]
When \(t>0\), every application is valid because the largest
index to which a shift is applied is
\[
 (q-1)+(t-1)=m-3.
\]
When \(t=0\), the displayed identity is simply \(C=E_q\).

Since \(q+t=m-1\) and
\(D^{m-1}=\Delta_{m-1}^{2}\) is central in \(\Br_{m-1}\),
we obtain
\[
 D^qCD^{-q}
 =
 D^{q+t}E_qD^{-(q+t)}
 =
 D^{m-1}E_qD^{-(m-1)}
 =
 E_q.
\]
Multiplying this equality on the left by \(D^{-q}\) yields
\begin{equation}\label{eq:transport}
 CD^{-q}=D^{-q}E_q.
\end{equation}

We can now perform the closure moves.
Using \eqref{eq:isolation} and a cyclic conjugation in
\(\Br_m\), we have
\[
 \operatorname{cl}_m(B\delta_m^{-q})
 =
 \operatorname{cl}_m(BCaD^{-q})
 \iso
 \operatorname{cl}_m(D^{-q}BCa).
\]
All three factors \(D^{-q}\), \(B\), and \(C\) belong to
\(\Br_{m-1}\). Hence, setting \(X=D^{-q}BC\), the last braid is
\[
 Xa=X\sigma_{m-1}^{-1},
 \qquad X\in\Br_{m-1}.
\]
It therefore admits a negative Markov destabilisation:
\[
 \operatorname{cl}_m(D^{-q}BCa)
 \iso
 \operatorname{cl}_{m-1}(D^{-q}BC).
\]
This is the only step that changes the number of strands.

Finally, working entirely in \(\Br_{m-1}\), cyclic conjugation
and \eqref{eq:transport} give
\begin{align*}
 \operatorname{cl}_{m-1}(D^{-q}BC)
 &\iso
 \operatorname{cl}_{m-1}(BCD^{-q})\\
 &=
 \operatorname{cl}_{m-1}(BD^{-q}E_q)\\
 &\iso
 \operatorname{cl}_{m-1}(E_qBD^{-q}).
\end{align*}
Substituting \(D=\delta_{m-1}\) proves
\eqref{eq:one-strand}.
All conjugations and the negative Markov destabilisation
preserve the oriented link type.
\end{proof}

\begin{corollary}\label{cor:negative-reduction}
For \(B\in\Br_u\) and \(2\leq q\leq u<P\),
\begin{equation}\label{eq:negative-reduction}
 \cl{B\delta_P^{-q}}
 \iso
 \cl{E_q^{P-u}B\delta_u^{-q}},
\end{equation}
where the closures on the left and right are taken on \(P\)
and \(u\) strands, respectively.
The isotopy can be realised using exactly \(P-u\) negative
Markov destabilisations, together with cyclic conjugations
and braid equalities.
\end{corollary}

\begin{proof}
For each integer \(j\) with \(0\leq j\leq P-u\), set
\[
 B_j=E_q^jB.
\]
Since \(q\leq u\), both \(E_q\) and \(B\) are supported on
the first \(u\) strands. Consequently,
\[
 B_j\in\Br_u
 \qquad (0\leq j\leq P-u).
\]

Suppose that \(0\leq j<P-u\), and put \(m=P-j\).
Then \(m\geq u+1\), so
\[
 B_j\in\Br_u\subseteq\Br_{m-1},
 \qquad
 2\leq q\leq u<m.
\]
Thus, all hypotheses of Lemma~\ref{lem:one-strand} hold.
Applying that lemma gives
\begin{align*}
 \operatorname{cl}_{P-j}(B_j\delta_{P-j}^{-q})
 &\iso
 \operatorname{cl}_{P-j-1}
   (E_qB_j\delta_{P-j-1}^{-q})\\
 &=
 \operatorname{cl}_{P-j-1}
   (B_{j+1}\delta_{P-j-1}^{-q}),
\end{align*}
where the equality uses \(E_qB_j=B_{j+1}\).

Iterating this step for
\(j=0,1,\ldots,P-u-1\) reduces the strand number successively
from \(P\) to \(u\). Since \(B_0=B\) and
\(B_{P-u}=E_q^{P-u}B\), the resulting isotopy is
\[
 \operatorname{cl}_P(B\delta_P^{-q})
 \iso
 \operatorname{cl}_u(E_q^{P-u}B\delta_u^{-q}),
\]
as required.

Each application of Lemma~\ref{lem:one-strand} uses exactly
one negative Markov destabilisation. There are \(P-u\)
applications, and all remaining operations are cyclic
conjugations and braid equalities. Hence the reduction uses
exactly \(P-u\) negative Markov destabilisations.
\end{proof}

The reduction in \eqref{eq:negative-reduction} was established
by a geometric argument in
\cite[Proposition 3.1]{dePaivaLiu}.
Here, Lemma~\ref{lem:one-strand} and the iteration above give
an algebraic proof using braid equalities, cyclic conjugations,
and negative Markov destabilisations.
We include this proof to make the braid-reduction arguments
underlying our construction self-contained, with the support
conditions and the number of destabilisations verified explicitly.

\subsection{An arbitrary product of negative prefix blocks}
\label{subsec:prefix}

We next apply the reduction in reverse to an ordered list of blocks.
We allow the widths to increase or decrease, as both possibilities
occur when we start from a general signed braid word.

\begin{proposition}\label{prop:prefix}
Let \(B\in\Br_p\), and let \(u_j,v_j\in\mathbb Z\) satisfy
\(2\leq u_j\leq p\) and \(v_j\geq1\), for \(1\leq j\leq m\). Put
\[
 V_0=0,\qquad V_j=v_1+\cdots+v_j,\qquad H=V_m.
\]
Then
\begin{equation}\label{eq:prefix}
 \cl{E_{u_1}^{v_1}\cdots E_{u_m}^{v_m}B}
 \iso
 \cl{B\delta_p^{u_1}
  \left(\prod_{j=1}^{m-1}
    \delta_{p+V_j}^{u_{j+1}-u_j}\right)
    \delta_{p+H}^{-u_m}}.
\end{equation}
The braid on the right has \(p+H\) strands and reduces to the braid on
the left by exactly \(H\) negative destabilisations.
\end{proposition}
\begin{proof}
We start with the braid on the right. Applying
Corollary~\ref{cor:negative-reduction} to its last block reduces the
number of strands from \(p+V_m\) to \(p+V_{m-1}\) and places
\(E_{u_m}^{v_m}\) at the beginning. The two final powers are now
adjacent, and their exponents add to
\[
 (u_m-u_{m-1})-u_m=-u_{m-1}.
\]
We repeat this procedure. Before the reduction indexed by \(j\), all
preceding factors are supported on at most \(p+V_{j-1}\) strands, and
each \(E\)-block already extracted is supported on the first \(p\)
strands. Hence
\[
 2\leq u_j\leq p\leq p+V_{j-1}<p+V_j
\]
verifies the hypotheses at every step. We place each extracted block
at the beginning, so the final order is
\(E_{u_1}^{v_1}\cdots E_{u_m}^{v_m}\).
At the last reduction the powers \(\delta_p^{u_1}\) and
\(\delta_p^{-u_1}\) cancel. The number of destabilisations is
\(\sum_j v_j=H\). A zero difference in \eqref{eq:prefix} contributes
an identity factor and causes no difficulty.
\end{proof}

\section{Cyclic residues of a braid word}\label{sec:cyclic}

In this section we write the input word as negative prefix blocks,
working in \(\Br_p\) until we apply the prefix formula in the next
section. Combining the ends of the word before taking residues gives
one residue per letter and the height bound \(H\leq pc\).

\subsection{The cyclic factorisation}\label{subsec:factorisation}

We start by writing the signed letters in a common form. Throughout
this section, assume that \(p\geq3\) and \(c\geq1\).
For each input letter, define
\begin{equation}\label{eq:letter-data}
 b_j=\frac{1-\epsilon_j}{2},\qquad
 a_j=i_j-b_j,\qquad t_j=1-b_j,\qquad
 q_j=\begin{cases}p-1,&\epsilon_j=1,\\2,&\epsilon_j=-1.\end{cases}
\end{equation}
We can then write both identities of Lemma~\ref{lem:letters} as
\begin{equation}\label{eq:unified-letter}
 \sigma_{i_j}^{\epsilon_j}
  =E_p^{a_j}E_{q_j}E_p^{-a_j-t_j}.
\end{equation}
Set \(a_{c+1}=a_1\) and
\begin{equation}\label{eq:cyclic-gaps}
 \gamma_j=a_{j+1}-a_j-t_j\qquad(1\leq j\leq c).
\end{equation}
Multiplying the identities in \eqref{eq:unified-letter} and combining
the outer powers across the cyclic boundary gives
\begin{equation}\label{eq:cyclic-equality}
 E_p^{-a_1}\beta E_p^{a_1}
 =E_{q_1}E_p^{\gamma_1}\cdots E_{q_c}E_p^{\gamma_c}.
\end{equation}
This is an exact equality of braids.

The bounds on \(a_j\) and \(t_j\) imply
\(-p\leq\gamma_j<p\). Define
\begin{equation}\label{eq:residues}
 h_j=\begin{cases}\gamma_j+p,&\gamma_j<0,\\
                   \gamma_j,&\gamma_j\geq0,\end{cases}
 \qquad F=\#\{j:\gamma_j<0\}.
\end{equation}
Then \(0\leq h_j<p\). By \eqref{eq:full-twist}, each negative gap
contributes one central positive full twist. Therefore
\begin{equation}\label{eq:U}
 E_p^{-a_1}\beta E_p^{a_1}=U\delta_p^{pF},\qquad
 U=E_{q_1}E_p^{h_1}\cdots E_{q_c}E_p^{h_c}.
\end{equation}
To compute each residue, we compare the gap with zero and add \(p\)
when necessary. Thus, this step requires no Euclidean division.

We delete the zero-height factors in \(U\) and combine adjacent factors
of the same width by adding their heights. We also combine equal
widths across the cyclic boundary. This is allowed by conjugation,
since \(\delta_p^{pF}\) is central in \(\Br_p\). We obtain a cyclic list
\begin{equation}\label{eq:root-list}
 (u_1,v_1),\ldots,(u_m,v_m),\qquad
 2\leq u_j\leq p,\quad v_j>0.
\end{equation}
When \(m>1\), consecutive widths are distinct, including the closing
pair. All cyclic rotations of this list give closures isotopic to
\(\cl\beta\) after the factor \(\delta_p^{pF}\) is restored.
Explicitly, if \(U=XY\) and \(Z=\delta_p^{pF}\), then
\[
 X^{-1}(UZ)X=YZX=YXZ,
\]
because \(Z\) is central. This also justifies moving a terminal block
to the beginning before merging it with the initial block.

\subsection{Height and descents}\label{subsec:height}

We now bound the total height and the number of descents of the cyclic
list. These quantities will determine the size of the presentation in
the next section.
Put \(n_-=\sum_j b_j\). Telescoping \eqref{eq:cyclic-gaps} gives
\[
 \sum_{j=1}^c\gamma_j=-\sum_{j=1}^c t_j=-(c-n_-).
\]
Hence the total height of the cyclic list is
\begin{equation}\label{eq:H}
 H:=\sum_{j=1}^m v_j
   =c+\sum_{j=1}^c h_j=pF+n_-\leq pc.
\end{equation}
The last inequality follows directly from \(h_j\leq p-1\).
There are at most two nonzero blocks per input letter before merging, so
\begin{equation}\label{eq:m-bound}
 m\leq2c.
\end{equation}

For \(m>1\), define the number of cyclic descents by
\begin{equation}\label{eq:K}
 K=\#\{j:u_j>u_{j+1}\},\qquad u_{m+1}=u_1.
\end{equation}
Since the cyclic list is nonconstant, it has at least one descent.
Each descent ends at a block of width less than \(p\), and this block
contains at least one of the \(c\) original factors \(E_{q_j}\): every
other original block has width \(p\). The blocks contain disjoint sets
of original factors. Therefore,
\begin{equation}\label{eq:K-bound}
 1\leq K\leq c.
\end{equation}
This argument also covers \(p=3\), when both signs have inner width
\(q_j=2\).

\begin{remark}\label{rem:boundary}
The cyclic boundary must be combined before \eqref{eq:residues} is
evaluated. Reducing the initial and terminal powers separately can leave
an \(E_p\)-block with height at least \(p\). That block contains a
negative full twist which cancels a positive full twist already extracted.
Equation~\eqref{eq:cyclic-equality} removes this redundancy directly.
\end{remark}

\section{Explicit parameters and the best cyclic cut}\label{sec:parameters}

In this section we find the parameters from the cyclic list and prove
that they represent the input closure. We minimise the final strand
parameter over cyclic cuts, using one division followed by additions
and comparisons.

\subsection{The direct parameter formula}\label{subsec:formula}

We first fix a starting point for the list. Choose a cut after
a cyclic descent and relabel so that \(u_m>u_1\), when \(m>1\).
Set
\[
 V_j=\sum_{i=1}^j v_i,\qquad R=p+H.
\]
Let \(\kappa\) be the number of internal descents. Thus
\(\kappa=K-1\) when \(m>1\), and \(\kappa=0\) when \(m=1\).
Write their positions as
\(j_1<\cdots<j_{\kappa}\), and put \(S_0=0\) and
\begin{equation}\label{eq:r-descents}
 r_\nu=p+V_{j_\nu},\qquad
 S_\nu=\sum_{i=1}^\nu r_i,
 \qquad 1\leq\nu\leq\kappa.
\end{equation}
When \(\kappa=0\), the descent lists are empty.

\begin{theorem}\label{thm:formula}
With the preceding notation, \(\cl\beta\) is represented by
the following list, in the displayed order:
\begin{equation}\label{eq:direct-formula}
\begin{aligned}
 &(p,pF+u_1),\\
 &\left(p+V_j,\ u_{j+1}-u_j
       +(p+V_j)\ind_{\{u_j>u_{j+1}\}}\right)_{j=1}^{m-1},\\
 &\begin{cases}
   (R,-u_m),&\kappa=0,\\[2pt]
   (R,r_1-u_m),\quad
   (R+S_\nu,r_{\nu+1}-r_\nu)_{\nu=1}^{\kappa-1},\quad
   (R+S_{\kappa},-r_{\kappa}),&\kappa>0.
  \end{cases}
\end{aligned}
\end{equation}
It is an offset-zero generalised \(T\)-link presentation.
If \(m>1\), its pair count and final strand parameter are exactly
\begin{equation}\label{eq:exact-size}
 k=m+K,\qquad
 M=p+H+\sum_{\nu=1}^{K-1}r_\nu.
\end{equation}
If \(m=1\), they are \(k=2\) and \(M=p+H\).
In both cases \(-M<s_k\leq-2\). From this output braid to
a conjugate of the input, the proof uses exactly \(M-p\)
negative destabilisations.
\end{theorem}

\begin{proof}
By Section~\ref{sec:cyclic}, the input closure is isotopic to
\(\cl{E_{u_1}^{v_1}\cdots E_{u_m}^{v_m}\delta_p^{pF}}\).
Apply Proposition~\ref{prop:prefix}, with base width \(p\)
and \(B=\delta_p^{pF}\). This gives the signed intermediate braid
\begin{equation}\label{eq:seed}
 \delta_p^{pF+u_1}
 \left(\prod_{j=1}^{m-1}
   \delta_{p+V_j}^{u_{j+1}-u_j}\right)
 \delta_R^{-u_m}\in\Br_R.
\end{equation}
The widths in this braid increase strictly. If there are no
internal descents, all intermediate exponents are positive,
so \eqref{eq:seed} already gives the required presentation.

Suppose now that \(\kappa>0\). At a descent, write
\(z_\nu=u_{j_\nu}-u_{j_\nu+1}\). Then
\[
 1\leq z_\nu\leq p-2<r_\nu.
\]
We separate a full twist by writing
\begin{equation}\label{eq:separation}
 \delta_{r_\nu}^{-z_\nu}
  =E_{r_\nu}^{r_\nu}\delta_{r_\nu}^{r_\nu-z_\nu}.
\end{equation}
The full twist \(E_{r_\nu}^{r_\nu}\) commutes with all preceding
factors, since they are supported on at most \(r_\nu\) strands.
We extract these full twists in increasing order of \(\nu\),
placing each immediately after those already extracted.
Thus, \eqref{eq:seed} becomes
\begin{equation}\label{eq:second-prefix}
 E_{r_1}^{r_1}\cdots
 E_{r_{\kappa}}^{r_{\kappa}}A_+\delta_R^{-u_m},
\end{equation}
where \(A_+\) consists of the initial and intermediate factors
in \eqref{eq:seed}, with every negative exponent increased
by its width. In particular, the factors of \(A_+\) have
precisely the first two lines of \eqref{eq:direct-formula},
and all their exponents are positive.

Apply Proposition~\ref{prop:prefix} again, this time with base
width \(R\), prefix widths
\(r_1<\cdots<r_{\kappa}<R\), prefix heights \(r_\nu\),
and \(B=A_+\delta_R^{-u_m}\in\Br_R\). The result is
\begin{equation}\label{eq:second-output}
 A_+\delta_R^{r_1-u_m}
  \left(\prod_{\nu=1}^{\kappa-1}
     \delta_{R+S_\nu}^{r_{\nu+1}-r_\nu}\right)
  \delta_{R+S_{\kappa}}^{-r_{\kappa}}.
\end{equation}
The support conditions for this second application hold because
\[
 2\leq r_\nu<R\leq R+S_{\nu-1}<R+S_\nu.
\]
Also \(r_1>p\geq u_m\), so the exponent at width \(R\) is
positive. The later intermediate exponents are positive
differences of strictly increasing widths. The last exponent
is \(-r_{\kappa}\leq-2\). This proves
\eqref{eq:direct-formula} and all sign and order conditions.

For \(m>1\), there are \(m\) pairs before the last line of
the formula, and that line contributes \(\kappa+1\) pairs,
giving \(m+\kappa+1=m+K\). The final width is
\(R+S_{\kappa}\). For \(m=1\), the list is simply
\[
 ((p,pF+u_1),(p+H,-u_1)).
\]
The final negative exponent has absolute value
\(u_m\leq p<M\) when \(\kappa=0\), and
\(r_{\kappa}<R<M\) when \(\kappa>0\).
Finally, the first prefix reduction removes \(H\) strands,
and the second, when needed, removes \(S_{\kappa}\) strands.
Their sum is \(M-p\), and all the destabilisations are negative.
\end{proof}

The proof applies the prefix formula twice, but the algorithm
computes \eqref{eq:direct-formula} directly. Thus, we do not
need to store an intermediate signed parameter list or the
expanded output braid.

\subsection{Optimising the cut}\label{subsec:cut}

We next find the cut which minimises \(M\) in \eqref{eq:exact-size}.
We keep the blocks of the cyclic list fixed and compare their possible
starting points. No further braid relations are used in this
minimisation.

Regard the list as a circle of total height \(H\). A descent location is
the boundary immediately after a block whose width exceeds the next
width. If \(x\) is the chosen descent location, the other \(K-1\)
descent locations occur at positive clockwise distances \(D_x(y)\).
Equation~\eqref{eq:exact-size} becomes
\begin{equation}\label{eq:cut-score}
 M_x=Kp+H+\sum_{y\neq x}D_x(y).
\end{equation}
This includes \(K=1\), when the sum is empty.

It suffices to consider cuts after descents. Indeed, a cut at any other
block boundary leaves all \(K\) descents internal. If the first descent
occurs at distance \(a>0\), moving the cut to that descent removes its
contribution \(p+a\) and subtracts \(a\) from the distance to every
remaining descent. The strand parameter decreases by \(p+Ka\), and
the pair count decreases by one. Thus a descent cut dominates every
non-descent cut in both quantities.

\begin{proposition}\label{prop:cut-bound}
For \(m>1\), the best descent cut satisfies
\begin{equation}\label{eq:average-bound}
 M\leq Kp+\frac{K+1}{2}H.
\end{equation}
Consequently the output of Algorithm~\ref{alg:core} satisfies
\(k\leq3c\) and \(M\leq pc(c+3)/2\).
\end{proposition}
\begin{proof}
For each unordered pair of distinct descent locations, its two directed
distances sum to \(H\). Averaging \eqref{eq:cut-score} over the \(K\)
possible cuts therefore gives
\[
 \frac1K\sum_x M_x
 =Kp+H+\frac{\binom K2H}{K}
 =Kp+\frac{K+1}{2}H.
\]
The minimum is at most the average. Using \(K\leq c\), \(H\leq pc\),
and \(m\leq2c\), we obtain
\[
 M\leq cp+\frac{c+1}{2}pc=\frac{pc(c+3)}2,
 \qquad k=m+K\leq3c.
\]
When \(m=1\), \(k=2\leq3c\) and
\(M=p+H\leq p(c+1)\leq pc(c+3)/2\), so the same bounds hold.
\end{proof}

\subsection{An additive best-cut computation}\label{subsec:additive-cut}

We now compute the minimum in a single scan of the list. The recurrence
below uses additions and comparisons at each candidate cut, avoiding
repeated products of growing integers and giving the stated bit bound.

In the stored cyclic order, let
\[
 0\leq x_0<x_1<\cdots<x_{K-1}<H
\]
be the height coordinates of the descent locations, and let
\(S=\sum_{\ell=0}^{K-1}x_\ell\). They are computed by cumulative
addition along the root list. For a cut at \(x_\ell\), the locations with
smaller coordinates wrap once around the circle. Hence
\begin{equation}\label{eq:cut-coordinate}
 M_\ell=Kp+H+S+\ell H-Kx_\ell.
\end{equation}
It remains to minimise \(\ell H-Kx_\ell\).
Perform the single division
\[
 H=KQ+A,\qquad 0\leq A<K.
\]
Write \(z_\ell=\lfloor\ell H/K\rfloor\) and
\(b_\ell=\ell H-Kz_\ell\). Then
\begin{equation}\label{eq:lex-score}
 \ell H-Kx_\ell=K(z_\ell-x_\ell)+b_\ell,
 \qquad 0\leq b_\ell<K.
\end{equation}
Minimising the right side is exactly lexicographic minimisation of
\((z_\ell-x_\ell,b_\ell)\). Starting from \(z_0=b_0=0\), update
\begin{equation}\label{eq:carry}
 \begin{aligned}
 z&\leftarrow z+Q,& b&\leftarrow b+A;\\[-2pt]
 &\text{if }b\geq K,& b&\leftarrow b-K,\quad z\leftarrow z+1.
 \end{aligned}
\end{equation}
By induction using the division algorithm, these updates give
\(z_{\ell+1},b_{\ell+1}\) at the next step. Thus, the loop uses only
additions, subtractions, and comparisons. Among equal minima choose the cut with
smallest stored block index; this makes the algorithm deterministic.
After the scan, \eqref{eq:cut-coordinate} can be evaluated with a
constant number of products, or the width can be read from the emitted
parameters.

\begin{algorithm}\label{alg:core}
For an expanded word in \(\Br_p\), with \(p\geq3\) and \(c\geq1\),
perform the following steps.
\begin{enumerate}
\item Compute the letter data and cyclic residues from
\eqref{eq:letter-data}--\eqref{eq:residues}, and count \(F\).
Append each block \((q_j,1)\), followed by \((p,h_j)\).
Omit zero heights and merge equal successive widths.
\item Merge equal first and last widths, if present. This gives the
cyclic list \eqref{eq:root-list} and its height \(H\).
\item If \(m>1\), find its descent cuts and use
\eqref{eq:lex-score}--\eqref{eq:carry} to choose the minimum. Rotate the
list to start immediately after that descent. If \(m=1\), retain its
single block.
\item Scan the rotated list, accumulating \(V_j\) and recording the
internal descent widths \(r_\nu\). Emit the pairs
\eqref{eq:direct-formula} in their displayed order.
\end{enumerate}
\end{algorithm}

\begin{corollary}\label{cor:core-correct}
Algorithm~\ref{alg:core} always terminates and returns the presentation
in Theorem~\ref{thm:formula}, with the bounds of
Proposition~\ref{prop:cut-bound}. Its cut minimises the final strand
parameter among all block-boundary cuts in this construction.
\end{corollary}
\begin{proof}
Each scan is over a finite list of at most \(2c\) blocks. The preceding
factorisation, prefix formula, and cut calculation prove that the
computed parameters give the required presentation. Cuts away from
descents give larger values, as shown in Section~\ref{subsec:cut}.
\end{proof}

\section{Elementary simplification and direct recognition}
\label{sec:shortcuts}

Some input words can be simplified or already give a much smaller
presentation. In this section we describe deterministic rules for
recognising several such cases before applying Algorithm~\ref{alg:core}. We justify
each rule by braid identities or Markov moves. These rules do not
require a search through braid relations.

\subsection{Cyclic cancellations and endpoint destabilisations}
\label{subsec:simplification}

We first simplify the input word, allowing its starting point to
change. This treats the word as circular, since a cyclic rotation is a
conjugation.
For a nonempty input put
\[
 s=\min_j i_j-1.
\]
Replace each signed index \(\epsilon_j i_j\) by
\(\epsilon_j(i_j-s)\), retaining the declared number \(p\).
This replacement is valid: iterating \eqref{eq:shift} backwards gives
\[
 E_p^{-s}\sigma_i^{\pm1}E_p^s=\sigma_{i-s}^{\pm1}
 \qquad(i\geq s+1).
\]
Thus, the uniform index shift is realised by conjugation and retains
all unused strands.

After the shift, apply the following rules, repeating after each deletion:
\begin{enumerate}
\item Cancel cyclically adjacent inverse letters. If the pair crosses the
chosen boundary, first rotate the word to put the pair next to each other
in a linear representative.
\item If \(\sigma_{p-1}^{\pm1}\) occurs exactly once, rotate that
letter to the end and perform a right Markov destabilisation of its sign.
Decrease \(p\) by one.
\item Otherwise, if \(\sigma_1^{\pm1}\) occurs exactly once, perform
the corresponding left destabilisation. Decrease \(p\) by one and
decrease every remaining unsigned index by one.
\end{enumerate}
In the second and third rules, the endpoint indices are those of the
current declared braid. If unused strands remain to the right of the
largest used generator, we cannot destabilise that generator.

For completeness, the left rule also follows from right destabilisation.
The automorphism
\(\rho_p(\sigma_i)=\sigma_{p-i}\) is conjugation by \(\Delta_p\).
If \(A\) uses indices \(2,\ldots,p-1\), then
\(\rho_p(A)\sigma_{p-1}^{\epsilon}\) admits a right destabilisation.
Applying \(\rho_{p-1}\) afterwards changes an original index \(i\)
to \(i-1\). Both reflections preserve the closure, so this is the
claimed left rule, for either sign.

To fix a deterministic order, cancel all currently queued inverse pairs,
then prefer a right destabilisation over a left one, and return to the
cancellation queue after a destabilisation. Any fixed order of the initial
queue suffices; we use increasing original letter position. At termination
write the working input as \((p',\beta')\), with length \(c'\).
Then \(p'\leq p\), \(c'\leq c\), and
\(\cl{\beta'}\iso\cl\beta\). These rules give sufficient conditions for simplification. If none
applies, this does not prove that the braid is minimal.

\subsection{A linear implementation of the deletion rules}
\label{subsec:deletion-data}

We now implement the deletion rules without scanning the whole word
after every deletion. The data structure below follows the specified
order in linear total work, including when \(p\) is much larger than
\(c\).

Assign each original letter an identifier in
\(\{0,\ldots,c-1\}\). Store the surviving letters in a circular doubly
linked list, with a live/deleted flag and a distinguished head. Queue all
initial adjacencies. When a letter is deleted, queue the newly exposed
adjacency at its predecessor. An entry concerning a deleted letter is
ignored. Each deletion creates at most one entry. Thus the number of
queue insertions, and the number of adjacency checks, is \(O(c)\).
Ordinary inverse cancellation removes two entries; a destabilisation
removes one.

Keep the original unsigned indices after the uniform shift. A left
destabilisation changes a lazy offset \(\lambda\); a surviving original
index \(i\) then means current index \(i-\lambda\).
Let \(P\) be the strand count at the start of this deletion stage, and
let \(r\) be the number of right destabilisations already made. Its
current left endpoint generator has original index \(\lambda+1\),
and its current right endpoint generator has original index \(P-r-1\).

At most \(c\) destabilisations can occur. Therefore only the original
generator indices in
\[
 \{1,\ldots,\min(c,P-1)\}\quad\text{or}\quad
 \{\max(1,P-c),\ldots,P-1\}
\]
can ever become eligible endpoint generators. For each of these two
zones store two arrays: the count of live occurrences of an index and
the bitwise XOR of their letter identifiers. If a count is one, the XOR
is exactly the identifier of its unique occurrence. On deleting a letter,
update the arrays for every zone to which its original index belongs;
overlapping zones are updated twice, once in each copy.

The four arrays have total length \(O(c)\). They are initialised in one
pass, and each deletion changes only a constant number of entries.
Endpoint eligibility is tested by a constant number of array accesses.
The head is rotated by a pointer change, and all surviving indices are
translated by the lazy offset in one final scan. This gives \(O(c)\)
integer operations and \(O(c)\) stored integers, with no hash table and
no loop through the unused strand indices. Recording the identifiers
deleted at each step also gives a short, independently checkable record
of the elementary simplification.

\subsection{Literal recognition of twisting words}\label{subsec:recognition}

In this subsection we recognise words already written in certain
twisting forms. We compare the Artin letters up to cyclic rotation.
These tests recognise the given presentations; they do not decide
whether an arbitrary word represents a link in one of these families.

For an all-positive word, first check whether \(1\) occurs. If it does
not, the positive recognition test fails and we proceed to Algorithm~\ref{alg:core}.
Otherwise, rotate to an occurrence of \(1\). Split before each occurrence
of \(1\), and check that each resulting block is exactly
\[
 1,2,\ldots,r-1
\]
for some width \(r\). A missing or repeated intermediate index causes
failure. Let \(r^{(1)},\ldots,r^{(t)}\) be the resulting cyclic width
list. It can be rotated into nondecreasing order precisely when it has
at most one cyclic strict descent. Cut after that descent if it exists,
and combine equal successive widths. The result is an ordinary positive
\(T\)-word. Every test and merge is a linear scan; no sorting is required.

For an all-negative word, test whether a cyclic rotation is
\(\delta_{p'}^{-a}\), with \(a\geq2\). A nonempty such word must contain
\(-(p'-1)\); begin there and check repeated blocks
\[
 -(p'-1),-(p'-2),\ldots,-1.
\]
The length must be a multiple of \(p'-1\). This test inspects only
the supplied letters; it does not generate a block of length \(p'-1\)
in advance.

Words containing both positive and negative letters are passed
to the arithmetic conversion in Algorithm~\ref{alg:core}.
The recognition tests above together use \(O(c')\) integer
operations.

\subsection{Unused strands and the boundary cases}\label{subsec:padding}

We next include the declared strand count and the normalisation
\(|s_k|\geq2\) in the direct returns. The following identity allows us
to do this while keeping the offset zero.

\begin{lemma}\label{lem:padding}
For \(p\geq2\) and \(B\in\Br_p\),
\begin{equation}\label{eq:padding}
 \cl{B\delta_p^{2p}\delta_{2p}^{-p}}\iso\cl B.
\end{equation}
The reduction from the left to the right uses \(p\) negative
destabilisations, followed by braid equality and conjugacy.
\end{lemma}
\begin{proof}
In Proposition~\ref{prop:prefix} take one prefix block, of width \(p\)
and height \(p\), with base braid \(B\delta_p^p\). It gives
\[
 \cl{B\delta_p^{2p}\delta_{2p}^{-p}}
 \iso\cl{E_p^pB\delta_p^p}=\cl B,
\]
because \(E_p^p=\delta_p^{-p}\) is central in \(\Br_p\).
This is also the case \(k=2\) of
\cite[Corollary 3.3]{dePaivaLiu}.
\end{proof}

If a recognised positive \(T\)-word \(B\) uses all \(p'\) declared
strands and has final exponent at least two, return it directly.
Otherwise use \eqref{eq:padding}. Append \((p',2p')\) and
\((2p',-p')\), merging the first appended pair with the previous one
if it already has width \(p'\). This produces at most two extra pairs
and exactly \(2p'\) strands. In particular, unused strands of \(B\)
remain part of the represented link.

An empty working word on \(q\geq2\) strands is the \(q\)-component
unlink and is returned as
\begin{equation}\label{eq:unlink}
 T((q,2q),(2q,-q)),
\end{equation}
by Lemma~\ref{lem:padding} with \(B=1\).
For the empty one-strand word we return
\begin{equation}\label{eq:unknot}
 T((2,2),(3,-2)).
\end{equation}
Indeed, Proposition~\ref{prop:prefix} with base width two, one block
\(E_2\), and \(B=1\), reduces \(\delta_2^2\delta_3^{-2}\) to
\(E_2=\sigma_1^{-1}\). One further negative destabilisation gives the
one-strand unknot.

If the nonempty working word has two strands, cyclic free reduction
leaves letters of one sign. It is \(\sigma_1^e\). A single remaining
letter would have been destabilised, so \(|e|\geq2\), and we return
\(T((2,e))\). A nonempty valid word on one strand is impossible.

\subsection{The complete algorithm}\label{subsec:wrapper}

We now combine the preceding simplifications and recognition tests with
the arithmetic conversion. This gives the complete algorithm in
Theorem~\ref{thm:main}.

\begin{algorithm}\label{alg:complete}
Given a positive declared strand number and a valid expanded signed
Artin word, proceed as follows.
\begin{enumerate}
\item For nonempty input, make the uniform index shift in
Section~\ref{subsec:simplification}. Perform cyclic inverse cancellations
and signed endpoint destabilisations, in the specified deterministic
order, using Section~\ref{subsec:deletion-data}.
\item If the working word is empty or has two strands, return the
presentation specified in Section~\ref{subsec:padding}.
\item Apply the literal cyclic tests of Section~\ref{subsec:recognition}.
On success, return the recognised presentation, with padding if required.
\item On failure, the working input has \(p'\geq3\), \(c'\geq1\).
Apply Algorithm~\ref{alg:core} to \((p',\beta')\) and return its output.
\end{enumerate}
\end{algorithm}

\begin{proposition}\label{prop:wrapper-size}
Algorithm~\ref{alg:complete} preserves the oriented input closure and
returns an offset-zero presentation satisfying Definition~\ref{def:gt}.
For initial \(p\geq3\) and \(c\geq1\), it satisfies
\eqref{eq:main-size}.
\end{proposition}
\begin{proof}
The preprocessing moves preserve the closure by
Section~\ref{subsec:simplification}. We proved the direct returns and
the padding identity above, and the remaining case follows from
Corollary~\ref{cor:core-correct}. Since each deletion removes at least
one letter, the preprocessing terminates.

For the arithmetic branch, apply Proposition~\ref{prop:cut-bound} to
\((p',c')\). Since \(p'\leq p\) and \(c'\leq c\), its bounds imply
\eqref{eq:main-size}. A direct nonempty return without padding has at
most \(c'\) pairs and at most \(p'\) strands. A padded positive return
has at most \(c'+2\leq3c'\leq3c\) pairs and \(2p'\leq2p\) strands.
Since \(c(c+3)/2\geq2\), this also satisfies the strand bound.
An empty return has two pairs and at most \(\max\{3,2p'\}\leq2p\)
strands, and again satisfies both bounds. This covers every branch.
\end{proof}

\begin{remark}\label{rem:no-search}
One may supplement the algorithm with braid-relation searches or other
certified simplifications. No such search is part of
Algorithm~\ref{alg:complete}, and no search budget is hidden in the
complexity statements below. The direct recognition tests may return a
positive final exponent. In contrast, the arithmetic branch always has
one negative final exponent, with \(-M<s_k\leq-2\).
\end{remark}

\section{Complexity and size of the output}\label{sec:complexity}

In this section we prove the complexity bounds in
Theorem~\ref{thm:main}. We first count the operations needed to compute
the binary parameters. We then calculate the writhe, the expanded
crossing count, and the number of Markov moves used in the proof.

\subsection{Parameter computation}\label{subsec:time}

We start by counting the operations at each stage of the algorithm,
including the direct returns. Throughout this subsection, we use the
model fixed in Section~\ref{subsec:model}.

\begin{proposition}\label{prop:complexity}
For \(p\geq3\), \(c\geq1\), Algorithm~\ref{alg:complete} computes its
binary output in \(O(c)\) integer operations, \(O(cL+L^2)\) bit
operations, and \(O(cL)\) bits of storage. Its binary output has
\(O(cL)\) bits. The arithmetic branch alone has the same bounds.
\end{proposition}
\begin{proof}
We first validate the input, shift the minimum index, and simplify the
word using the linked list. These steps take \(O(c)\) operations by
Section~\ref{subsec:deletion-data}. The recognition tests are linear
in the surviving length, and padding adds only a constant number of
arithmetic operations and pairs.

For the arithmetic branch, we compute the quantities in
\eqref{eq:letter-data}--\eqref{eq:residues} in one cyclic scan of the
working letters. We retain the first letter to compute the last gap.
During the same scan, we accumulate \(F,H,n_-\), the input writhe,
and the merged root list. Each letter adds at most two blocks, and
merging the cyclic boundary takes a constant number of list operations.
Thus, the resulting list has at most \(2c\) entries.

In the next scan, we record the descent coordinates and their sum.
The best-cut recurrence uses one division \(H=KQ+A\), followed by
additions and comparisons for each candidate. We perform the cyclic
rotation by changing an index or by copying the list once. Finally,
we compute \eqref{eq:direct-formula} by cumulative additions. Each
internal descent contributes one stored width and at most one extra
output pair, so all these lists have length \(O(c)\).

There are \(O(c)\) additions, comparisons, Boolean operations, and
memory operations. The only growing-integer products or divisions needed
for the parameters are a constant number: \(pF\), the one quotient and
remainder for the cut, and, if requested, evaluation of its final score.
Products by the constants two and three can of course be implemented
by additions. Even evaluating the a priori bounds introduces only a
constant number of further products.

Every root width is at most \(p\), every height is at most \(pc\),
every descent width is at most \(p+H\), and every final exponent in
the arithmetic branch has absolute value at most \(p+H\). Thus
\begin{equation}\label{eq:integer-bounds}
 |s_j|\leq p(c+1),\quad
 r_j\leq\frac{pc(c+3)}2,
 \quad S\leq KH\leq pc^2.
\end{equation}
The other scores and counters have the same polynomial bounds, up to
a fixed power. Their bit lengths are \(O(L)\), since
\(\log_2(pc^2+1)=O(L)\). Direct-return exponents and widths are bounded
by \(O(c+p)\), so the same conclusion holds in those branches.

The simple operations therefore take \(O(cL)\) bit operations, and
the constant number of multiplications and divisions takes
\(O(L^2)\). Storing the \(O(c)\) integers requires \(O(cL)\) bits.
There are \(O(c)\) output pairs, each consisting of \(O(L)\)-bit
integers, and the same bound includes the delimiters. This proves
the proposition.
\end{proof}

\begin{proof}[Proof of Theorem~\ref{thm:main}]
Proposition~\ref{prop:wrapper-size} proves correctness and the size
bounds. Proposition~\ref{prop:complexity} proves the complexity bounds
under the stated encoding. Empty inputs and strand numbers one and two
are handled explicitly in Section~\ref{subsec:padding}. Their costs
are covered uniformly by
\(O((c+1)L+L^2)\) bit operations and \(O((c+1)L)\) bits of storage,
with \(L\) defined as in \eqref{eq:L} for the given input.
\end{proof}

For the stated routine the number of integer operations is
\(\Theta(c)\) on nonempty inputs, since it always inspects all letters.
The bit bound is an upper bound; an \(\Omega(cL)\) lower bound is not
asserted for every individual input. For fixed \(p\), it becomes
\(O(c\log(c+1))\). Decimal serialisation and expansion of powers into
Artin letters are separate output tasks.

\subsection{Expanded crossings and Markov moves}\label{subsec:expanded}

We now calculate the writhe and crossing count of the braid given by
the parameters. The following identities allow us to compute both
quantities without a separate product for each output pair.

\begin{proposition}\label{prop:writhe}
Let \(\mathcal P\) be the output of Algorithm~\ref{alg:core} on the
input \((p,\beta)\), let \(M\) be its final strand parameter, and let
\(w=\sum_j\epsilon_j\). Then
\begin{equation}\label{eq:writhe}
 w(\mathcal P)=w-(M-p).
\end{equation}
Writing its last pair as \((M,s_k)\),
\begin{equation}\label{eq:crossing-shortcut}
 C(\mathcal P)=w-(M-p)-2(M-1)s_k.
\end{equation}
In particular,
\begin{equation}\label{eq:expanded-bound}
 C(\mathcal P)=O(p^2c^3).
\end{equation}
\end{proposition}
\begin{proof}
Theorem~\ref{thm:formula} reduces the output to a conjugate of the input
by exactly \(M-p\) negative destabilisations. A negative
destabilisation deletes one negative crossing and increases writhe by
one. Braid equality and conjugation preserve writhe. This proves
\eqref{eq:writhe}.

Only the last output exponent is negative. Replacing its signed
contribution \((M-1)s_k\) by its absolute value changes the total by
\(-2(M-1)s_k\). This proves \eqref{eq:crossing-shortcut}.
Since \(|w|\leq c\), \(|s_k|\leq p+H\leq p(c+1)\), and
\(M=O(pc^2)\), the claimed bound follows.
\end{proof}

The same crossing bound holds for the complete algorithm when
\(p\geq3,c\geq1\). In a direct return without padding the expanded
length is at most \(c\). Padding adds
\[
 2p'(p'-1)+p'(2p'-1)=4(p')^2-3p'
\]
letters to its working word, and empty returns have \(O(p^2)\)
letters. The arithmetic branch is bounded by
Proposition~\ref{prop:writhe} with \(p'\leq p,c'\leq c\).

We can also compute these statistics within the same parameter bound.
For any valid output whose only possibly negative exponent is the last,
\begin{equation}\label{eq:general-crossing-shortcut}
 C(\mathcal P)=
 \begin{cases}
 w(\mathcal P)-2(M-1)s_k,&s_k<0,\\
 w(\mathcal P),&s_k>0.
 \end{cases}
\end{equation}
For the arithmetic branch, use \eqref{eq:writhe}; a literal direct
return preserves the working writhe, and padding lowers it by \(p'\).
The unlink and unknot formulas have writhe \(-q\) and \(-2\),
respectively. Thus computing these statistics requires only a constant
number of additional growing-integer products after the input scan.

The proof provides a compressed description of a Markov certificate.
For the arithmetic branch it has \(M-p\) negative destabilisations,
which is \(O(pc^2)\), rather than \(O(c)\) in general. Writing every
intermediate braid, or expanding every braid equality used in the
certificate, is a different computation. We make no linear-time claim
for that task.

\subsection{Sharpness for the construction}\label{subsec:sharpness}

We now show that the arithmetic construction attains both size bounds
and give an example of quadratic strand growth for the complete
algorithm. These examples also distinguish the expanded braid length
from the time needed to compute its parameters.

\begin{proposition}\label{prop:sharp-core}
For every \(p\geq3\) and \(c\geq1\), the arithmetic algorithm on
\(\beta=\sigma_1^c\in\Br_p\) has
\[
 k=3c,\qquad M=\frac{pc(c+3)}2.
\]
Thus both uniform bounds are sharp for Algorithm~\ref{alg:core}.
\end{proposition}
\begin{proof}
Every letter has \(a_j=t_j=1\), so
\(\gamma_j=-1\), \(h_j=p-1\), and \(F=c\). The cyclic roots are
\[
 ((p-1,1),(p,p-1))
\]
repeated \(c\) times. Hence \(m=2c\), \(K=c\), and \(H=pc\).
All descent locations are equally spaced by height \(p\), so every
eligible cut gives the same score. After one is chosen, the internal
descent widths are
\(2p,3p,\ldots,cp\). Therefore
\begin{align*}
 M&=p+pc+p(2+3+\cdots+c)\\
  &=\frac{pc(c+3)}2,
 \qquad k=m+K=3c.
\end{align*}
For \(c=1\), the descent list is empty and the same formula gives
\(M=2p\).
\end{proof}

The complete algorithm recognises these positive powers, or destabilises
them when a unique endpoint occurs. Thus Proposition~\ref{prop:sharp-core}
does not establish sharpness of the numerical strand bound for the
complete algorithm.

The complete algorithm can still produce quadratically many strands
after all the stated simplifications. Consider
\begin{equation}\label{eq:quadratic-family}
 \beta=(\sigma_1\sigma_3)^n\in\Br_4,
 \qquad n\geq2,\qquad c=2n.
\end{equation}
There is no inverse pair, and each endpoint generator occurs \(n\)
times. The literal positive test fails: every proposed block beginning
with \(1\) immediately meets \(3\) instead of \(2\).
For an odd-position letter the gap is \(1\); for an even-position
letter it is \(-3\). Both residues are \(1\). The roots are consequently
\((3,1),(4,1)\), repeated \(c\) times. Thus
\[
 H=2c,\quad m=2c,\quad K=c,\quad
 r_\nu=4+2\nu\quad(1\leq\nu<c).
\]
The resulting exact sizes are
\begin{equation}\label{eq:quadratic-size}
 k=3c,\qquad
 M=4+2c+\sum_{\nu=1}^{c-1}(4+2\nu)=c^2+5c.
\end{equation}
This proves quadratic worst-case strand growth for this complete
routine at fixed \(p\). It is not a lower bound for all conversion
algorithms. In fact, the input link is the split union of two torus
links, and its special structure suggests other possible simplifications.

\section{Examples}\label{sec:examples}

In this section we apply the algorithm to several explicit braid
words. We show the cyclic residue calculation and both uses of the
prefix formula, and give examples where direct recognition shortens
the output.

\subsection{The figure-eight knot}\label{subsec:figure-eight}

We start with the figure-eight knot, given by the standard
three-strand word
\[
 \beta=(\sigma_1\sigma_2^{-1})^2.
\]
The four \(a_j\)'s are all one, while
\((t_1,t_2,t_3,t_4)=(1,0,1,0)\). Hence
\[
 (\gamma_j)=(-1,0,-1,0),\qquad
 (h_j)=(2,0,2,0),\qquad F=2.
\]
Both signs have inner width two. After merging, the cyclic root list is
\[
 ((2,2),(3,2),(2,2),(3,2)),
 \qquad H=8,\quad K=2.
\]
Both descent cuts give the same minimum. We choose the displayed
one and apply the prefix formula to obtain
\[
 \delta_3^8\delta_5\delta_7^{-1}\delta_9\delta_{11}^{-3}.
\]
There is one internal descent, of width \(r_1=7\).
Using \(\delta_7^{-1}=E_7^7\delta_7^6\), the second application gives
\[
 \delta_3^8\delta_5\delta_7^6\delta_9\delta_{11}^4\delta_{18}^{-7}.
\]
Consequently
\begin{equation}\label{eq:figure-eight-parameters}
 \cl\beta\iso
 T((3,8),(5,1),(7,6),(9,1),(11,4),(18,-7)).
\end{equation}
The output has six pairs and eighteen strands. Its reduction uses
\(18-3=15\) negative destabilisations. Since the input writhe is zero,
\[
 w(\mathcal P)=-15,\qquad
 C(\mathcal P)=-15+2\cdot17\cdot7=223.
\]
This computation gives an explicit presentation; we do not claim that
it is minimal.

\subsection{A destabilisation and an unused strand}
\label{subsec:small-shortcuts}

We next illustrate a destabilisation and the need to retain an unused
strand. First, take
\(\sigma_1^3\sigma_2^{-1}\in\Br_3\). Its unique occurrence of the
right endpoint generator admits a negative destabilisation. The
remaining two-strand word is \(\sigma_1^3\), so the complete routine
returns \(T((2,3))\).

Now take \(\sigma_1^4\in\Br_3\). This has an unused third strand,
whose closure is a split unknot. The word is recognised as a positive
\(T\)-word, but its largest width is two, smaller than the declared
strand count. Padding gives
\begin{equation}\label{eq:unused-example}
 T((2,4),(3,6),(6,-3)).
\end{equation}
It has three components: the two components of \(T(2,4)\), together
with the unused-strand unknot. Its expanded length is
\(4+12+15=31\). Returning \(T((2,4))\) alone would lose a component
and would not solve the stated input problem.

\subsection{A six-strand example}\label{subsec:six-strand}

We now apply the algorithm to a positive word with several blocks
which is not already written as a \(T\)-word:
\begin{equation}\label{eq:six-input}
 \beta=(\sigma_1\sigma_2\sigma_3\sigma_4\sigma_5)
       (\sigma_2\sigma_3\sigma_4)
       (\sigma_1\sigma_2\sigma_4\sigma_5)\in\Br_6.
\end{equation}
There are \(c=12\) letters. The cyclic gaps and residues are
\[
 \begin{aligned}
 (\gamma_j)&=(0,0,0,0,-4,0,0,-4,0,1,0,-5),\\
 (h_j)&=(0,0,0,0,2,0,0,2,0,1,0,1).
 \end{aligned}
\]
Thus \(F=3\) and \(H=18\). One stored root list is
\[
 ((5,5),(6,2),(5,3),(6,2),(5,2),(6,1),(5,2),(6,1)).
\]
Its descent coordinates are \(0,7,12,15\), so \(K=4\) and
\(S=34\). Equation~\eqref{eq:cut-coordinate} gives strand scores
\[
 76,\quad66,\quad64,\quad70.
\]
The minimum starts at height twelve. The rotated root list is
\[
 ((5,2),(6,1),(5,2),(6,1),(5,5),(6,2),(5,3),(6,2)).
\]
The internal descent widths are \(9,12,19\), and \(R=24\).
The direct formula gives
\begin{equation}\label{eq:six-output}
\begin{split}
 T\bigl(&(6,23),(8,1),(9,8),(11,1),(12,11),(17,1),\\
        &(19,18),(22,1),(24,3),(33,3),(45,7),(64,-19)\bigr).
\end{split}
\end{equation}
There are twelve pairs and sixty-four strands. From
\eqref{eq:writhe}--\eqref{eq:crossing-shortcut},
\[
 w(\mathcal P)=12-(64-6)=-46,\qquad
 C(\mathcal P)=-46+2\cdot63\cdot19=2348.
\]
We compute the twelve parameter pairs directly, without writing
these 2348 letters.

\section{Starting from a link diagram}\label{sec:diagrams}

In this section we start with a link diagram and first put it in braid
form. We then apply Algorithm~\ref{alg:complete}. We specify the
diagram encoding, recall Yamada--Vogel braiding, and bound the cost of
finding and performing its moves and reading the final braid word.
This includes the work needed between moves, in addition to their
number.

\subsection{Diagram size and Yamada--Vogel braiding}
\label{subsec:vogel}

We first describe the input and recall the braiding theorem. The input
is an oriented combinatorial diagram on the sphere, given by a plane
rotation system with crossing information and component orientations.
We also record the placement of disconnected projection components
and any crossing-free circles. Labels have
\(O(\log(N+1))\) bits. We do not include the problem of recognising
crossings from a raster image or manipulating geometric coordinates.
Let
\begin{equation}\label{eq:N}
 n=c(D),\qquad \mu=\#\pi_0(L(D)),\qquad N=n+\mu,
\end{equation}
where \(n\) counts crossings of the supplied diagram and \(\mu\)
counts components of the represented link \(L(D)\). For a nonempty
link \(N\geq1\).

Oriented smoothing gives \(s\) Seifert circles, with
\(s\leq n+\mu\): each smoothing changes the number of curve components
by at most one. For two circles on the sphere, their complementary
annulus determines whether their orientations agree in its first
homology. Call them coherent if they agree and incoherent otherwise.
Let \(h(D)\) be the number of unordered incoherent pairs. In particular,
\begin{equation}\label{eq:h-bound}
 0\leq h(D)\leq\binom{s}{2}.
\end{equation}

We use the following standard form of the Yamada--Vogel theorem
\cite{Yamada,Vogel}; its reducing-arc proof is also given in
\cite[Section 2.2]{BirmanBrendle}. If \(h(D)>0\), there is an arc
joining an incoherent pair whose interior avoids the Seifert picture
(the circles together with the signed arcs recording crossings).
A Reidemeister II move along it preserves \(s\), introduces two
crossings, and lowers \(h\) by one. Repetition ends at height zero,
where a braid on \(s\) strands can be read with one Artin letter per
remaining crossing. It follows that, for this diagram, the resulting
braid length is at most
\begin{equation}\label{eq:vogel-length}
 n+2h(D)\leq n+s(s-1).
\end{equation}
For our estimates, \eqref{eq:h-bound} is sufficient. We therefore do
not need a sharper bound with an additional connectedness convention.

\subsection{A polynomial implementation bound}\label{subsec:vogel-time}

We now bound the time needed to find reducing arcs and to read the
final braid word from the combinatorial diagram. We allow the data to
be reconstructed after each move. This gives a direct polynomial
bound, without asserting that this implementation of the classical
braiding algorithm is optimal.

First make the projection connected if necessary. Let \(t\) be its
number of connected components, including crossing-free circles.
Then \(t\leq\mu\). A face incident to two distinct projection
components admits a Reidemeister II finger move between them. Such a
face exists whenever the projection is disconnected: on the sphere,
the regions and projection components have a connected incidence
structure. The move creates two crossings, connects those two
projection components, and preserves the oriented link. Repeat at most
\(t-1\) times. Face-boundary and component traversals find the next
pair in polynomial time, or in \(O(N)\) work per step by rebuilding
their lists. Thus this preliminary stage takes \(O(N^2)\) word
operations and gives a connected diagram \(D_0\) with
\begin{equation}\label{eq:connected-diagram}
 n_0\leq n+2(\mu-1),\qquad
 s_0\leq n_0+\mu\leq3N.
\end{equation}
The single crossing-free unknot is handled directly, without this stage.

Suppose the current connected diagram has \(q\) crossings. Store its
oriented half-edges, the cyclic order at each crossing, and the crossing
signs. Smooth each crossing combinatorially while retaining a marked
connecting arc. Traversing the smoothed curves and the face boundaries
of this Seifert picture requires \(O(q+s_0)\) word operations. The
picture has \(O(q+s_0)\) edges and half-edges; connectedness ensures
that its face boundaries are available from its rotation system.

We find a reducing arc by the following linear scan. On each incidence
of an oriented circle segment with a face, record its circle identifier
and whether the face lies to the left or to the right of that oriented
segment. For each face and each of these two sides, retain the first
circle identifier seen. If another incidence on that side has a
different identifier, the two circles are incoherent and can be joined
inside the face. Indeed, they lie on the boundary of their common
annular side, and equal side designations give opposite homology
orientations in that annulus. Conversely, an incoherent pair exposed
in one face has equal side designations. The reducing-arc theorem
therefore guarantees that this scan finds a pair whenever \(h>0\).
The incidences are visited once, so the scan costs \(O(q+s_0)\).
Inserting the two-crossing move is an update of the half-edge rotation
system; rebuilding and relabelling the whole diagram is allowed within
the same bound. No coordinates for the interior arc are required.

The number of reducing moves is at most \(\binom{s_0}{2}=O(N^2)\).
Throughout the computation,
\begin{equation}\label{eq:q-bound}
 q\leq n_0+s_0(s_0-1)=O(N^2).
\end{equation}
Even allowing \(O(q+s_0^2)=O(N^2)\) work per reconstruction, the
total reducing stage therefore takes \(O(N^4)\) word operations.
We do not need to list all incoherent pairs or recompute \(h\) during
the scan. The existence of a reducing arc and the strict decrease of
height prove termination.

At height zero, read the Artin word from the coherent Seifert picture.
One combinatorial method orders the circles from inside to outside,
choosing the point at infinity suitably, and joins their cyclic crossing
lists successively. A crossing between consecutive circles numbered
\(i,i+1\) is replaced by \(\sigma_i\) or \(\sigma_i^{-1}\) according
to its sign. Lists between consecutive common crossings are inserted
as the child circle is joined to its parent. This is the circle-list
extraction described in \cite[Section 4.8]{Bartholomew}.
There are \(q\) crossing symbols and at most \(s_0\) circle lists.
Each crossing is emitted once. Even if each replacement or splice
scans or copies a whole current list, its length is \(O(q+s_0)\),
and there are \(O(q+s_0)\) such replacements and joins. This gives
the conservative extraction bound \(O((q+s_0)^2)=O(N^4)\).
It returns a braid with
\begin{equation}\label{eq:braid-from-diagram}
 p=s_0=O(N),\qquad c\leq n_0+s_0(s_0-1)=O(N^2).
\end{equation}
All labels remain polynomially bounded in \(N\). Consequently each
word operation above has \(O(\log(N+1))\) bit cost in the model of
Section~\ref{subsec:model}. No optional braid simplification is needed
in this braiding stage.

\begin{theorem}\label{thm:diagram}
From a combinatorial diagram \(D\) of a nonempty oriented link, with
\(n\) crossings and \(\mu\) components, one can compute an offset-zero
generalised \(T\)-link presentation of the same oriented link in
\[
 O(N^4)\text{ word operations},\qquad
 O(N^4\log(N+1))\text{ bit operations},\qquad N=n+\mu.
\]
The output has \(O(N^2)\) pairs, \(O(N^5)\) strands, and
\(O(N^2\log(N+1))\) bits of binary parameters. The parameter-conversion
stage after braiding takes \(O(N^2\log(N+1))\) bit operations.
\end{theorem}
\begin{proof}
By the preceding analysis, Yamada--Vogel braiding has the stated word
and bit bounds and gives \eqref{eq:braid-from-diagram}. We apply
Theorem~\ref{thm:main} to this braid, using the explicit boundary
returns when \(p<3\) or \(c=0\). In the nonempty arithmetic case,
we have
\[
 L=O(\log(N+1)),\quad
 k\leq3c=O(N^2),\quad
 M\leq pc(c+3)/2=O(N^5).
\]
Proposition~\ref{prop:complexity} gives the parameter-stage bit cost
and output size. The elementary boundary cases satisfy these
asymptotic bounds as well. Both stages preserve the oriented link,
so their composition proves the theorem.
\end{proof}

For knots, \(\mu=1\), and the theorem gives a bound polynomial in
\(n+1\). For general links, \(n\) alone cannot bound the size of the
input or output: consider the \(\mu\)-component unlink with a
crossing-free diagram. Also, the \(n\) in the theorem is the number
of crossings of the supplied diagram. Replacing it by the minimum
crossing number of the link would require a separate minimisation
problem, which is not part of the conversion.

\begin{remark}\label{rem:diagram-scope}
The fourth-degree word bound is a conservative bound for the
reconstruction scheme just described. It is not an assertion of
optimality for Yamada--Vogel braiding, and it is not obtained merely
by counting reducing moves. The conversion in
Sections~\ref{sec:cyclic}--\ref{sec:shortcuts} takes an Artin braid word
as input. The diagram stage in this section composes that conversion
with the classical braiding algorithm.
\end{remark}

\section{Polynomial-time transfer for link properties and invariants}
\label{sec:transfer}

In this section we use the conversion to relate computations on link
diagrams to computations on generalised \(T\)-link parameters. We
consider properties with a yes-or-no answer and invariants with a
specified output encoding. We first compare the sizes of the two
descriptions and prove that they can be translated in polynomial time.
We then prove the transfer statements for properties, finitely encoded
invariants, and real-valued invariants with a prescribed error. The
algorithms are deterministic and uniform throughout; in particular,
they must work for arbitrarily many pairs. We use the standard
principle of polynomial reductions; see, for example,
\cite[Chapter 2]{AroraBarak}. The output bounds proved above allow us
to apply this principle here.

We consider polynomial-time solvability here. A standard multitape
machine can simulate the bit-cost random-access computation with
polynomial overhead by storing the visited cells in a list and
scanning it for each access. Polynomial bit cost bounds the number
of cells and the lengths of their addresses and contents, so the
simulation remains polynomial. The precise conversion bounds use
the model of Section~\ref{subsec:model}.

\subsection{Parameter size and polynomial translations}
\label{subsec:transfer-sizes}

We start by specifying the sizes used in the transfer. Our algorithm
computes a binary parameter list, while the reverse construction
writes its Artin expansion to obtain a diagram. We keep track of both
sizes before comparing the corresponding problems.

Let \(\mathscr G\) be the set of all admissible presentations in
Definition~\ref{def:gt}. For \(\mathcal P\in\mathscr G\), write
\(M=r_k\), and retain the expanded length
\(C(\mathcal P)\) from \eqref{eq:metrics-definition}. Define
\begin{equation}\label{eq:transfer-sizes}
 W(\mathcal P)=M+C(\mathcal P),\qquad
 U(\mathcal P)=k+\sum_{j=1}^k(r_j+|s_j|).
\end{equation}
Thus \(U\), up to constant factors, is the length of the list when
the positive integers are written in unary, with separate signs and
delimiters. The binary length, denoted by \(\ell(\mathcal P)\), satisfies
\begin{equation}\label{eq:transfer-binary-size}
 \ell(\mathcal P)=\Theta\!\left(
 k+\sum_{j=1}^k\left(
 \left\lceil\log_2(r_j+1)\right\rceil+
 \left\lceil\log_2(|s_j|+1)\right\rceil\right)\right)
\end{equation}
for a fixed self-delimiting encoding of integer lists.

Every exponent is nonzero and every width is at least two. Consequently
\[
 k\leq C(\mathcal P),\qquad
 \sum_j|s_j|\leq C(\mathcal P),\qquad
 \sum_jr_j\leq2C(\mathcal P).
\]
For the last inequality, use
\(r_j\leq2(r_j-1)|s_j|\) term by term. In the other direction,
\[
 C(\mathcal P)
 \leq M\sum_j|s_j|\leq U(\mathcal P)^2.
\]
It follows that
\begin{equation}\label{eq:transfer-size-comparison}
 U(\mathcal P)\leq4W(\mathcal P),\qquad
 W(\mathcal P)\leq U(\mathcal P)+U(\mathcal P)^2.
\end{equation}
Polynomial time in \(W\) is therefore equivalent to ordinary
polynomial time on unary parameter lists. We call this the
\emph{expanded-size parameter problem}. This convention will be used
for each equivalence below. The stronger requirement of polynomial
time in the binary length \(\ell\) is treated separately in
Section~\ref{subsec:binary-classifications}.

Let \(\mathscr D\) be the combinatorial diagrams specified in
Section~\ref{sec:diagrams}. They describe nonempty oriented links. We
measure a diagram \(D\) by
\(N(D)=c(D)+\mu(D)\). An explicit encoding has polynomially related
bit length, since its labels use \(O(\log(N+1))\) bits and every
crossing-free component is recorded individually. Thus polynomial time
in \(N\) is also polynomial time in the ordinary encoded input size.
Statements about \(\mathscr D\) and \(\mathscr G\) take valid
descriptions as inputs.

\begin{proposition}\label{prop:translations}
There are deterministic maps
\[
 F:\mathscr D\longrightarrow\mathscr G,
 \qquad E:\mathscr G\longrightarrow\mathscr D
\]
with
\begin{equation}\label{eq:translation-types}
 T(F(D))\iso L(D),\qquad L(E(\mathcal P))\iso T(\mathcal P).
\end{equation}
The binary list \(F(D)\) is computed in
\(O(N^4\log(N+1))\) bit operations and satisfies
\begin{equation}\label{eq:translation-output}
 \ell(F(D))=O(N^2\log(N+1)),\qquad
 W(F(D))=O(N^8).
\end{equation}
The map \(E\) can be computed in
\(O(W\log(W+1))\) bit operations, with
\(N(E(\mathcal P))=O(W)\), where \(W=W(\mathcal P)\).
In particular, diagrams and unary parameter lists admit polynomial-time
translations in both directions.
\end{proposition}
\begin{proof}
For \(F\), first braid the diagram and then apply Algorithm~\ref{alg:complete}.
Theorem~\ref{thm:diagram} gives the time and binary-length bounds,
and \eqref{eq:braid-from-diagram} gives
\(p=O(N)\) and \(c=O(N^2)\). The output strand parameter is
\(M=O(N^5)\). Proposition~\ref{prop:writhe} and the complete-algorithm
discussion following it give
\[
 C(F(D))=O(p^2c^3)=O(N^8)
\]
when \(p\geq3\) and \(c\geq1\). The direct small-strand and empty-word
returns have polynomial expanded length as well: for empty words it
is \(O(p^2)\), and for at most two strands it is \(O(c+1)\).
Thus \(W(F(D))=O(N^8)\) in every case. The preservation of the
oriented link follows from the two stages of the construction. If a
unary output is requested, writing it takes additional time
\(O(U(F(D))\log(U(F(D))+1))=O(N^8\log(N+1))\), for example by
using binary counters, and remains polynomial by
\eqref{eq:transfer-size-comparison}.

For \(E\), expand each \(\delta_{r_j}^{s_j}\) as an Artin word;
negative powers use the reversed order of inverse generators. There
are exactly \(C(\mathcal P)\) crossing symbols. Construct the standard
braid diagram on \(M\) strands and close it by nested exterior arcs,
without adding crossings. Its crossing number as a diagram is
\(C(\mathcal P)\), and its number of components is at most \(M\).
Unused strands are retained as crossing-free components.

One can build the combinatorial diagram by keeping an array of the
current strand endpoints and a record of the half-edges at each
crossing. Each generator updates two endpoints and creates one
crossing record. Closing the strands and recording crossing-free
components take \(O(M)\) further operations. There are
\(O(C+M)=O(W)\) operations on labels of \(O(\log(W+1))\) bits.
This proves the bound for \(E\), as well as
\(N(E(\mathcal P))\leq C(\mathcal P)+M=W\), up to the fixed encoding
conventions. By \eqref{eq:transfer-size-comparison}, the same procedure
is polynomial on unary parameter input.
\end{proof}

\subsection{Properties as membership problems}
\label{subsec:property-transfer}

We first consider a property with a yes-or-no answer. We express a
complete classification as a membership criterion in the parameters
and compare its running time with that of deciding the property from
a diagram. Since the input is a presentation, different presentations
of the same oriented link must give the same answer.

Let \(\mathscr L\) denote the nonempty oriented link types in \(S^3\).
Fix a function
\[
 \chi:\mathscr L\longrightarrow\{0,1\}.
\]
Define the parameter set
\begin{equation}\label{eq:property-parameter-set}
 S_\chi=\{\mathcal P\in\mathscr G:
                   \chi(T(\mathcal P))=1\}.
\end{equation}
A complete classification gives a necessary and sufficient criterion
for membership in \(S_\chi\). We now compare the cost of testing this
criterion with the cost of deciding the property from a diagram.

\begin{theorem}\label{thm:decision-transfer}
The following are equivalent:
\begin{enumerate}
\item There is a uniform deterministic algorithm which, given
\(D\in\mathscr D\), returns \(\chi(L(D))\) in time polynomial in
\(N(D)\).
\item There is a uniform deterministic algorithm which, given
\(\mathcal P\in\mathscr G\), decides whether
\(\mathcal P\in S_\chi\) in time polynomial in \(W(\mathcal P)\).
\end{enumerate}
Equivalently, the decision problem on diagrams and the membership
problem on unary parameter lists are polynomial-time interreducible.
\end{theorem}
\begin{proof}
Suppose first that the parameter algorithm has running time at most
\(R(W)\), where \(R\) is a polynomial. On input \(D\), compute
\(\mathcal P=F(D)\) and run the parameter algorithm on this list.
The answer is correct because
\[
 \chi(L(D))=\chi(T(F(D))).
\]
For a fixed constant \(A\), the running time is at most
\begin{equation}\label{eq:decision-composition-cost}
 O(N^4\log(N+1))+R(A(1+N)^8),
\end{equation}
which is polynomial in \(N\). If the parameter algorithm expects
unary input, include its \(O(N^8\log(N+1))\) conversion and writing
cost. This remains polynomial.

Conversely, suppose the diagram algorithm has running time at most
\(Q(N)\), with \(Q\) polynomial. On input \(\mathcal P\), construct
\(E(\mathcal P)\) and apply the diagram algorithm. By
\eqref{eq:translation-types}, the answer is
\(\chi(T(\mathcal P))\). For a fixed constant \(B\), the running time
is at most
\[
 O(W\log(W+1))+Q(BW),
\]
which is polynomial in \(W\), and hence polynomial in the unary input
length. The two constructions are the required reductions.
\end{proof}

For example, we may take \(\chi\) to be unknot recognition, unlink
recognition, or hyperbolicity. The theorem then says that the diagram
problem is solvable in polynomial time exactly when the corresponding
parameter membership problem is. To obtain an algorithm from a
classification, we still need the membership criterion and its
running-time bound.

\subsection{Finitely encoded invariants and auxiliary inputs}
\label{subsec:invariant-transfer}

We now consider invariants whose values are encoded by finite strings.
We fix the encoding so that the diagram and the parameter list ask
for the same answer. This includes integer-valued invariants and
homology groups given by a specified table of degrees, ranks, and
torsion invariant factors. We include the time needed to write the
answer.

Let \(\Sigma\) be a fixed finite alphabet, and let
\[
 I:\mathscr L\longrightarrow\Sigma^*
\]
be a fixed invariant with the chosen encoding. Define
\(I_{\mathscr D}(D)=I(L(D))\) and
\(I_{\mathscr G}(\mathcal P)=I(T(\mathcal P))\).

\begin{theorem}\label{thm:invariant-transfer}
The function \(I_{\mathscr D}\) is computable in time polynomial in
\(N(D)\) if and only if \(I_{\mathscr G}\) is computable in time
polynomial in \(W(\mathcal P)\), equivalently in polynomial time on
unary parameter input. The algorithms in the two directions return
the same encoded value.
\end{theorem}
\begin{proof}
The identities
\[
 I_{\mathscr D}=I_{\mathscr G}\circ F,\qquad
 I_{\mathscr G}=I_{\mathscr D}\circ E
\]
follow from \eqref{eq:translation-types}, since the invariant depends
only on the link type. By Proposition~\ref{prop:translations}, both
translations and their output sizes are polynomially bounded. We
compose with the evaluation algorithm as in the preceding proof.
Neither translation changes or expands the encoded invariant value.
\end{proof}

Together, Theorems~\ref{thm:decision-transfer} and
\ref{thm:invariant-transfer} prove
Theorem~\ref{thm:intro-transfer}.
The output encoding remains part of the problem: we use the
same prescribed representation of the invariant value for
diagram and parameter inputs, and the running time includes
writing that representation. Tasks allowing several valid
answers are covered by Proposition~\ref{prop:auxiliary-transfer},
provided that the condition defining acceptable outputs depends
only on the oriented link type and the auxiliary input.

The same reasoning allows additional data, such as a degree in which
an invariant is requested. More generally, fix an isotopy-invariant
class \(\mathscr C\subseteq\mathscr L\) and an output condition
\(\mathcal R(L,z,y)\), where \(z,y\in\Sigma^*\). The condition
depends on \(L\) only through its link type. The task is to produce
an admissible \(y\) from a presentation of \(L\in\mathscr C\) and
an auxiliary input \(z\); any further existence promise is part of
the task.

\begin{proposition}
\label{prop:auxiliary-transfer}
The task defined by \(\mathcal R\) is solvable on diagram input in
time polynomial in \(N(D)+|z|\) if and only if it is solvable on
parameter input in time polynomial in \(W(\mathcal P)+|z|\).
The assertion holds with the promise \(L\in\mathscr C\).
\end{proposition}
\begin{proof}
We use the translations \((D,z)\mapsto(F(D),z)\) and
\((\mathcal P,z)\mapsto(E(\mathcal P),z)\). They preserve the
auxiliary string and the link type, and hence also the input promise
and the admissible outputs. Their time and output-size bounds are
polynomial in the respective input sizes. Composing with the
algorithm for either description proves the result. Since membership
in \(\mathscr C\) is an input promise, no recognition procedure for
this class is required.
\end{proof}

\subsection{Real-valued invariants and volume}
\label{subsec:real-transfer}

We next consider real-valued invariants, such as volume, with a
specified accuracy. We use absolute error and a requested number of
binary digits, asking for the same approximation from either a
diagram or a parameter list.

Fix a real-valued invariant
\(V:\mathscr C\longrightarrow\mathbb R\) on an isotopy-invariant
class \(\mathscr C\). A request of precision \(b\geq1\) asks for
an integer \(m\) such that
\begin{equation}\label{eq:real-accuracy}
 \left|2^{-b}m-V(L)\right|\leq2^{-b}.
\end{equation}
We return the signed binary integer \(m\), with denominator specified
by \(b\). The precision request is supplied in unary as \(1^b\).
Thus, polynomial time is measured in the presentation size and in
\(b\), rather than in \(\log b\).

\begin{corollary}\label{cor:real-transfer}
There is a uniform algorithm producing an approximation
\eqref{eq:real-accuracy} from \((D,1^b)\) in time polynomial in
\(N(D)+b\) if and only if there is such an algorithm from
\((\mathcal P,1^b)\) in time polynomial in \(W(\mathcal P)+b\).
Both statements have the promise that the represented link lies in
\(\mathscr C\).
\end{corollary}
\begin{proof}
Apply Proposition~\ref{prop:auxiliary-transfer} with \(z=1^b\)
and with \(\mathcal R(L,1^b,m)\) given by
\eqref{eq:real-accuracy}. Neither translation changes \(b\),
\(V(L)\), or the required error. Hence no loss of precision is
introduced by converting the presentation.
\end{proof}

For hyperbolic volume, let \(\mathscr C\) be the links with hyperbolic
complement and let \(V\) be their volume. A uniformly efficient
parameter formula would give an approximation of the same accuracy
from any diagram satisfying this promise. If we want a problem
defined on all links, we must specify a value or a distinguished
output for nonhyperbolic links and include it in the parameter
procedure. The transfer then applies to this total task. The
corollary assumes neither a volume formula nor an algorithm for
deciding hyperbolicity.

If a real invariant is represented exactly by some finite symbolic
encoding, Theorem~\ref{thm:invariant-transfer} applies to that
specified symbolic task. Such an encoding and its interpretation
must be fixed separately from the approximation convention above.

\subsection{Binary parameter lists and effective classifications}
\label{subsec:binary-classifications}

We finish this section with the case of a classification or formula
which is evaluated in time polynomial in the binary parameter length.
This gives a polynomial-time algorithm on diagrams. The hypothesis
on the parameter procedure is stronger than the expanded-size
hypothesis used in the preceding equivalences.

\begin{corollary}\label{cor:binary-transfer}
Let a fixed link property or a fixed finitely encoded link
invariant admit a uniform deterministic parameter procedure
with running time \(R(\ell(\mathcal P))\), where \(R\) is
polynomial. It then admits a polynomial-time algorithm on
diagram input. More generally, a task with auxiliary input
\(z\) that is solvable in time
\(R(\ell(\mathcal P)+|z|)\) is solvable from \((D,z)\)
in time polynomial in \(N(D)+|z|\). In particular, an
approximation algorithm polynomial in \(\ell(\mathcal P)+b\)
yields a diagram algorithm polynomial in \(N(D)+b\).

In each case, it suffices that the parameter procedure be
correct and satisfy the stated running-time bound on
\(F(\mathscr D)\), with the corresponding input promises.
\end{corollary}

\begin{proof}
Compute \(F(D)\) and then apply the parameter procedure,
retaining any auxiliary input. Take \(z\) to be the empty
string when there is no auxiliary input.
By \eqref{eq:translation-output}, the total time is bounded by
\begin{equation}\label{eq:binary-transfer-time}
 O(N^4\log(N+1))+
 R\bigl(A(1+N)^2\log(N+1)+|z|\bigr)
\end{equation}
for a fixed constant \(A\). This is polynomial in \(N+|z|\).
All lists passed to the procedure belong to \(F(\mathscr D)\).
Thus, correctness and the stated running-time bound are
required only on this image, with the corresponding input
promises. No procedure for recognising the image is needed.
For approximation, take \(z=1^b\).
\end{proof}

Corollary~\ref{cor:binary-transfer} proves
Corollary~\ref{cor:intro-binary-transfer}. It does not assert a
two-way equivalence for arbitrary binary parameter input. Indeed,
the admissible lists
\[
 \mathcal P_t=((2,2^t+1)),\qquad t\geq1,
\]
have \(\ell(\mathcal P_t)=O(t)\), whereas
\(C(\mathcal P_t)=2^t+1\). Their standard Artin expansions cannot
be written in polynomial time in \(t\). Thus the expansion used
for the converse in Theorem~\ref{thm:decision-transfer} is not a
polynomial translation for succinct binary input. A converse in
that model would require an additional argument.

For a property \(\chi\), an explicit classification might have the
form
\[
 \mathcal P\in S_\chi\quad\Longleftrightarrow\quad
 \Phi_\chi(r_1,s_1,\ldots,r_k,s_k).
\]
For an invariant, it might give an expression for
\(I(T(\mathcal P))\), or for an approximation to \(V(T(\mathcal P))\).
The classification establishes the correct answer. To determine the
running time, we must bound the cost of testing \(\Phi_\chi\) or
evaluating the invariant expression. Polynomially many operations on integers
of polynomially bounded bit length give a polynomial bound when
each operation itself has polynomial bit complexity. A compact
formula containing an exponential-size sum or an unbounded search
does not by itself give such a bound. A bound proved for each fixed
number of pairs must also be made uniform when \(k\) varies.

For an infinite table, we must specify how to find the entry for the
supplied parameters. Unit-cost access to arbitrary entries would
give an oracle model. A table with one representative per link type
also requires a way to identify that representative. Our conversion
does not assert canonicity, so the procedure must be correct for
every input in its domain, including different lists representing
the same link.

Thus, to obtain an algorithm for general diagrams, it remains to
find an efficiently testable classification or an efficiently
evaluable invariant formula in the parameters. It suffices to do
this on the output class of the conversion. Composing with \(F\)
then gives the required polynomial-time computation. The conversion
translates the presentation, and the classification or formula
determines the property or invariant.

\section*{Acknowledgements and use of AI tools}

The authors formulated the research problem, proposed the main
ideas and general approach, and guided the development and
refinement of the construction.

A large language model was used to assist with algebraic
calculations, explicit parameter transformations, and the
organisation and checking of intermediate steps. It also
assisted with refining the construction and its proofs,
checking the bounds on running time, storage, and presentation
size, preparing computational checks, and drafting and revising
the manuscript.

The authors retain full responsibility for the final
mathematical content, the verification of the arguments
and computations, and the accuracy of the references.

\bibliographystyle{plain}
\bibliography{Generalised_T_Links_26_Pages}
\end{document}